\documentclass[pdflatex,sn-mathphys-num]{sn-jnl}% Math and Physical Sciences Numbered Reference Style
\usepackage{graphicx}%
\usepackage{multirow}%
\usepackage{amsmath,amssymb,amsfonts}%
\usepackage{amsthm}%
\usepackage{mathrsfs}%
\usepackage[title]{appendix}%
\usepackage{xcolor}%
\usepackage{textcomp}%
\usepackage{manyfoot}%
\usepackage{booktabs}%
\usepackage{algorithm}%
\usepackage{algorithmicx}%
\usepackage{algpseudocode}%
\usepackage{listings}%
\theoremstyle{thmstyleone}%
\newtheorem{theorem}{Theorem}%  meant for continuous numbers
\newtheorem{proposition}[theorem]{Proposition}% 

\theoremstyle{thmstyletwo}%
\newtheorem{remark}{Remark}%

\theoremstyle{thmstylethree}%
\newtheorem{definition}{Definition}%

\theoremstyle{thmstyleone}%
\newtheorem{lemma}{Lemma}%
\newtheorem{assumption}{Assumption} %

\begin{document}

\title[A Proximal Gradient Framework for Composite Multiobjective Optimization]{A Proximal Gradient Framework for Composite Multiobjective Optimization on Riemannian Manifolds}

\author*{\fnm{Kangming} \sur{Chen}}\email{kangming@tmu.ac.jp}

\affil{\orgdiv{Faculty of Economics and Business Administration},
\orgname{Tokyo Metropolitan University},
\orgaddress{\city{Tokyo}, \postcode{192-0397}, \country{Japan}}}

\abstract{
% This paper proposes a Riemannian Multiobjective Proximal Gradient Method (RMPGM) for unconstrained composite optimization on manifolds. Avoiding prior scalarization, RMPGM  achieves global convergence to Pareto stationary points with an $\mathcal{O}(1/k)$ rate. We further develop two variants: an inexact RMPGM accommodating controlled computational errors, and a trust-region framework enjoying an $\mathcal{O}(\epsilon^{-2})$ iteration complexity. Numerical experiments on representative manifold problems demonstrate the efficiency and robustness of the proposed methods.
This paper proposes a Riemannian Multiobjective Proximal Gradient Method (RMPGM) for composite optimization problems on manifolds. Unlike scalarization-based approaches, the proposed framework directly handles vector-valued objectives and establishes global convergence to Pareto stationary points, together with an $\mathcal{O}(1/k)$ convergence rate. We further develop two variants to enhance practicality and performance: an inexact RMPGM that allows controlled inexactness in solving subproblems, and a trust-region RMPGM that adaptively adjusts the penalty parameter and achieves an $\mathcal{O}(\epsilon^{-2}) $iteration complexity. Numerical experiments demonstrate that the proposed methods are consistently outperform subgradient-based baselines.
}

\keywords{Riemannian optimization, Multiobjective optimization, Proximal gradient method, Trust-region method, Pareto stationarity, Composite optimization}
\pacs[MSC Classification]{ 90C30, 65K05, 49M37, 53B21}

\maketitle

%\section*{Update Notes}

%This version changes two important parts: fonts and bibliography.

%\textbf{Fonts}: Due to the newtx package updates, we changed the font settings for all the templates of ElegantLaTeX. Under \hologo{XeLaTeX}, we use \lstinline{fontspec} package to set the font to TeX Gyre Terms/Heros. 

%\textbf{Bibliography}: The bib file is no longer \lstinline{wpref.bib}, it's same with ElegantBook bibfile, \lstinline{reference.bib}. Besides, we use biblatex/biber rather than bibtex to handler bibliography, you can use bibstyle and citestyle to set the styles. For convenience, we offer a \lstinline{bibend} option, which can take values of \lstinline{biber} (default) and \lstinline{bibtex}, please refer to Bibliography section and biblatex package document for more information.

\section{Introduction}

Riemannian optimization offers a natural framework for problems involving geometric constraints. Translating constrained Euclidean formulations into unconstrained problems on manifolds generally allows algorithms to better leverage the curvature and symmetry of the search space. Substantial progress has been made in developing manifold optimization algorithms, including gradient-based methods, trust-region methods, and Newton-type methods~\cite{AbsMahSep2008, boumal2023intromanifolds, absil_trust-region_2007, chen2026riemannian, sato_riemannian_2021, weber2023riemannian}.

In parallel, multiobjective optimization studies problems involving multiple conflicting objectives and seeks Pareto optimal solutions. Descent-based methods play an important role in approximating the Pareto frontier. 
Notable developments include the multiobjective steepest descent method~\cite{fliege2000steepest, grana_drummond_steepest_2005}, inexact descent methods~\cite{fukuda_inexact_2013}. For composite problems in Euclidean spaces, proximal gradient methods~\cite{tanabe_proximal_2019} have proven effective, yielding established convergence properties, including a proven $\mathcal{O}(1/k)$ rate~\cite{tanabe_convergence_2023}.

Motivated by applications in machine learning and signal processing involving multiple competing objectives, such as sparse principal component analysis and low-rank matrix recovery on matrix manifolds~\cite{journee2010generalized, vandereycken2013low}, there is growing interest in extending composite optimization techniques to the Riemannian multiobjective setting.
The multiobjective steepest descent~\cite{bento_unconstrained_2012, bento_inexact_2013}, subgradient~\cite{bento_subgradient_2013}, conjugate gradient~\cite{chen2025nonlinear, najafiMultiobjectiveConjugateGradient2023}, proximal point methods~\cite{bento_proximal_2018} and other methods~\cite{eslami2023trust} have been extended to the Riemannian manifolds. For composite Riemannian multiobjective optimization problems of the form $F_i=f_i+g_i$ with nonsmooth regularizers $g_i$, existing approaches are still relatively limited and are mainly based on proximal point--type methods~\cite{bento_proximal_2018,bento_locally_2018,upadhyay_inexact_2025}. Such methods require solving implicit proximal subproblems at each iteration, which may become expensive in large-scale settings. For single-objective problems, Riemannian proximal gradient methods avoid this bottleneck by linearizing only the smooth component~\cite{chen2020proximal, huang_riemannian_2022}. This approach has been further extended to include inexact evaluations~\cite{huang_inexact_2023} and trust-region mechanisms~\cite{zhao_proximal_2023}. 
Adapting these techniques to the multiobjective case, however, introduces specific structural difficulties. 
Unlike the Euclidean setting, the nonsmooth subgradients belong to different tangent spaces along the iterates, preventing direct compactness and limiting arguments.
Moreover, the multiobjective proximal subproblem possesses a minimax structure, which complicates both the derivation of optimality conditions and the design of implementable inexact criteria.

Our main contributions are summarized as follows:
\begin{itemize}
    \item \textbf{Riemannian multiobjective proximal gradient framework:} We propose RMPGM for composite multiobjective optimization on Riemannian manifolds. We establish global convergence to Pareto stationary points together with an $\mathcal{O}(1/k)$ convergence rate.
    
    \item \textbf{Inexact variant:} We develop an inexact version of RMPGM based on a relative residual condition. Global convergence is preserved while allowing approximate solutions of the proximal subproblems.
    
    \item \textbf{Trust-region extension:} We introduce a trust-region type variant (TR-RMPGM) that adaptively updates the proximal parameter and does not require prior knowledge of the Lipschitz constant. We prove global convergence and an $\mathcal{O}(\epsilon^{-2})$ iteration complexity bound.

\end{itemize}

The remainder of this paper is organized as follows.
Section~\ref{sect: Preliminaries} reviews essential background on multiobjective optimization and Riemannian geometry. Section~\ref{sec: MRPGD method} presents the RMPGM algorithm and its global convergence analysis under $L$-retraction-smoothness, together with the $\mathcal{O}(1/k)$ convergence rate under retraction-convexity. Section~\ref{sec: Inexact method} develops the inexact scheme. Section~\ref{sec: trust region} introduces the trust-region variant with its convergence and complexity analysis. Numerical experiments are reported in Section~\ref{sec: expreriment}, and conclusions are drawn in Section~\ref{sec: Conclusion}.

\section{Preliminaries}
\label{sect: Preliminaries}
    In this paper, we consider the following unconstrained multiobjective optimization
    problem:
    \begin{equation}
    \label{eq:prob}
    %\tag{P}
    \begin{aligned}
        \min & \quad F(x) \\
        \mbox{s.t.} & \quad x \in \mathcal{M} ,
    \end{aligned}
    \end{equation}
    where $F \colon \mathcal{M} \rightarrow (\mathbb{R} \cup \{+\infty\})^m$ is a vector-valued function with $F:=\left(F_{1}, \ldots, F_{m}\right)^{\top}$, and $\top$ denotes transpose. We assume that each component $F_{i} \colon \mathcal{M} \rightarrow \mathbb{R} \cup \{+\infty\}$ is defined by
    $$
    F_{i}(x):=f_{i}(x)+g_{i}(x), \quad i=1, \ldots, m,
    $$
    where $f_{i} \colon \mathcal{M} \rightarrow \mathbb{R} \cup \{+\infty\}$ is proper, closed and continuously differentiable, and $g_{i} \colon \mathcal{M} \rightarrow \mathbb{R}  \cup \{+\infty\}$ is proper, closed and     geodesically convex but not necessarily differentiable.

Here we present some essential background on Riemannian manifolds \cite{AbsMahSep2008, boumal2023intromanifolds, sato_riemannian_2021}.
A Riemannian manifold $\mathcal{M}$ is a manifold endowed with a Riemannian metric 
$(\eta_x, \sigma_x) \mapsto \langle\eta_x, \sigma_x\rangle_x \in \mathbb{R}$, where $\eta_x$ and $\sigma_x$ are tangent vectors in the tangent space of $\mathcal{M}$ at $x$.
The tangent space of a manifold $\mathcal{M}$ at $x \in \mathcal{M}$ is denoted as $T_x \mathcal{M}$, 
and the tangent bundle of $\mathcal{M}$ is denoted as $T \mathcal{M}:=\left\{(x, \eta) \mid \eta \in T_x \mathcal{M}, x \in \mathcal{M}\right\}$. The norm of $\eta \in T_x \mathcal{M}$ is defined as $\|\eta\|_x:=\sqrt{\langle\eta, \eta\rangle_x}$. 
For a map $F: \mathcal{M} \rightarrow \mathcal{N}$ between two manifolds $\mathcal{M}$ and $\mathcal{N}, \mathrm{D} F(x): T_x \mathcal{M} \rightarrow T_{F(x)} \mathcal{N}$ denotes the derivative of $F$ at $x \in \mathcal{M}$. 
The Riemannian gradient $\operatorname{grad} f(x)$ of a function $f: \mathcal{M} \rightarrow \mathbb{R}$ at $x \in \mathcal{M}$ is defined as the unique tangent vector at $x$ satisfying $\langle\operatorname{grad} f(x), \eta\rangle_x = \mathrm{D} f(x)[\eta]$ for any $\eta \in T_x \mathcal{M}$.

Now, we define geodesic convexity sets and functions.

\begin{definition}\cite{boumal2023intromanifolds}
    A set $\mathcal{X}$ is called geodesically convex if for any $x, y \in \mathcal{X}$, there is a geodesic $\gamma$ with $\gamma(0)=x, \gamma(1)=y$ and $\gamma(t) \in \mathcal{X}$ for all $t \in[0,1]$.
\end{definition}

\begin{definition}
    A function $h \colon \mathcal{M} \rightarrow \mathbb{R}$ is called geodesically convex, where $\mathcal{S}\subseteq \mathcal{M}$ is a geodesically convex set, if for any $p, q \in \mathcal{M}$, we have $h(\gamma(t)) \leq(1-t) h(p)+t h(q)$ for any $t \in[0,1]$, where $\gamma$ is the geodesic connecting $p$ and $q$.
\end{definition}

In Riemannian optimization, we use retraction to project points from the tangent space of the manifold onto the manifold itself while preserving the underlying Riemannian metric. 

\begin{definition}\cite{sato_riemannian_2021}
    A smooth map  $R \colon T \mathcal{M} \rightarrow \mathcal{M}$ is called a retraction on a smooth manifold $\mathcal{M}$ if  the restriction of $R$  to the tangent space $T_x \mathcal{M}$ at any point $x \in \mathcal{M}$, denoted by $R_x$, satisfies  the following conditions:
    \begin{enumerate}
    \item $R_x\left(0_x\right)=x$,
    \item $\mathrm{D} R_x\left(0_x\right)=\operatorname{id}_{T_x \mathcal{M}}$ for all $x \in \mathcal{M}$, 
    \end{enumerate}
    where $0_x$ and $\operatorname{id}_{T_x \mathcal{M}}$ are the zero vector of $T_x \mathcal{M}$ and identity map in $T_x \mathcal{M}$, respectively.
\end{definition}
Retractions map tangent vectors back to the manifold $\mathcal{M}$, providing a computationally efficient alternative to the exponential map for maintaining iterate feasibility. They generalize the standard iterative update to the Riemannian setting as
$$x_{k+1}=R_{x_k}(t_k d _k),\quad \text{for }  k = 0,1,2, \ldots,$$
where $d_k$ is a descent direction and $t_k$ is a step size. 

Another key concept in Riemannian optimization is vector transport, which ensures computations preserve the manifold's intrinsic geometry. By using vector transport, algorithms can efficiently navigate the curved space of the manifold and compute geometric quantities necessary for convergence to optimal solutions.
\begin{definition}\cite{sato_riemannian_2021}
A map $\mathcal{T} \colon T \mathcal{M} \oplus T \mathcal{M} \rightarrow T \mathcal{M}$ is called a vector transport on $\mathcal{M}$ if it satisfies the following conditions, where $T \mathcal{M} \oplus T \mathcal{M}:=\left\{(\xi, d) \mid \xi, d \in T_x \mathcal{M}, x \in \mathcal{M}\right\}$ is the Whitney sum:
\begin{enumerate}
\item There exists a retraction $R$ on $\mathcal{M}$ such that $\mathcal{T}_d(\xi) \in T_{R_x(d)} \mathcal{M}$ for all $x \in \mathcal{M}$ and $\xi, d \in T_x \mathcal{M}$.
\item For any $x \in \mathcal{M}$ and $\xi \in T_x \mathcal{M}, \mathcal{T}_{0_x}(\xi)=\xi$ holds, where $0_x$ is the zero vector in $T_x \mathcal{M}$, i.e., $\mathcal{T}_{0_x}$ is the identity map.
\item For any $a, b \in \mathbb{R}, x \in \mathcal{M}$, and $\xi, d, \zeta \in T_x \mathcal{M}, \mathcal{T}_d(a \xi+b \zeta)=a \mathcal{T}_d(\xi)+b \mathcal{T}_d(\zeta)$ holds, i.e., $\mathcal{T}_d$ is a linear map from $T_x \mathcal{M}$ to $T_{R_x(d)} \mathcal{M}$.
\end{enumerate}
\end{definition}

Note that the map \( \mathcal{T} \) defined by \( \mathcal{T}_d(\xi) := \mathrm{P}_{\gamma_{x, d}}^{1 \leftarrow 0}(\xi) \) is also a vector transport, where \( \mathrm{P}_{\gamma_{x, d}}^{1 \leftarrow 0} \) denotes the parallel translation along the geodesic \( \gamma_{x, d}(t) := \operatorname{Exp}_x(t d) \), which connects \( \gamma_{x, d}(0) = x \) and \( \gamma_{x, d}(1) = \operatorname{Exp}_x(d) \), with the exponential map \( \operatorname{Exp} \) serving as a retraction.

In this paper, we consider the standard Pareto partial order induced by the non-negative orthant $\mathbb{R}_+^m$. For any two vectors $u, v \in \mathbb{R}^m$, we write $u \preceq v$ if $v - u \in \mathbb{R}_+^m$, and $u \prec v$ if $v - u \in \operatorname{int}(\mathbb{R}_+^m)$.
Based on this partial order, we formally define the optimality conditions directly on the Riemannian manifold $\mathcal{M}$. A point $x^* \in \mathcal{M}$ is called \emph{Pareto optimal} for $F$ if there exists no $x \in \mathcal{M}$ such that $F(x) \preceq F(x^*)$ and $F(x) \neq F(x^*)$. Similarly, $x^* \in \mathcal{M}$ is a \emph{weakly Pareto optimal} point if there is no $x \in \mathcal{M}$ satisfying $F(x) \prec F(x^*)$. It is well known that every Pareto optimal point is weakly Pareto optimal, but the converse is generally false. The set of all (weakly) Pareto optimal values constitutes the (weakly) Pareto frontier. 

A necessary condition for weak Pareto optimality is Pareto stationarity. In a manner analogous to  \cite{hosseini_line_2018} and \cite{bento_subgradient_2013}, we formally define Pareto stationarity on Riemannian manifolds as follows. A point $x^* \in \mathcal{M}$ is called \emph{Pareto stationary} (or \emph{critical}) if
\begin{equation}\label{stationary}
    \max_{i=1,\ldots,m} F_{i}^{\circ}(x^* ; \eta) \geq 0 \quad \text{for all } \eta \in T_{x^*}\mathcal{M},
\end{equation}
where $F_i^\circ(x^*; \eta)$ denotes the Clarke generalized directional derivative~\cite{clarke1990optimization, hosseini2011generalized}, which satisfies the relation $F_i^\circ(x^*; \eta) = \max \{ \langle \xi, \eta \rangle_{x^*} : \xi \in \partial F_i(x^*) \}$, and $\partial F_i(x^*)$ is naturally the Clarke subdifferential at $x^*$.

\begin{assumption}\label{assume0}
        There exists an open neighborhood $\mathcal{U} \supseteq \omega_{x_0}$ such that, for each $i=1,\ldots,m$, the function $g_i$ is locally Lipschitz on $\mathcal{U}$.
    \end{assumption}

    Under Assumption~\ref{assume0}, each composite objective $F_i = f_i + g_i$ is locally Lipschitz on $\mathcal{U}$. Consequently, the Clarke generalized directional derivative is well-defined and satisfies the exact sum rule:
    \begin{equation}\label{eq:clarke-sum-rule}
        F_i^\circ(x;\eta) = \langle \operatorname{grad} f_i(x), \eta \rangle_x + g_i^\circ(x;\eta), \quad \forall x \in \mathcal{U},\ \forall \eta \in T_x\mathcal{M}.
    \end{equation}
    Note that this assumption is automatically satisfied when $g_i$ is a finite-valued convex function on $\mathcal{M}$.

\section{Proximal gradient method for multiobjective
optimization on Riemannian manifolds}
\label{sec: MRPGD method}

    The proximal gradient method for multiobjective optimization on Riemannian manifolds is formalized in Algorithm \ref{alg:VRPGD}.
    It is worth noting that when the step size $t_k$ equals 1, we have $x_{k+1} = R_{x_k}(\eta_k)$. Consequently, the RMPGM update is obtained by solving the following subproblem in the tangent space $T_{x_k}\mathcal{M}$:
    \begin{equation}
    \label{eq:pg_step}
    \eta_k =\underset{\eta \in T_{x_k}\mathcal{M}}{\operatorname{argmin}} \left( \psi_{x_k}(\eta) + \frac{\tilde{L}}{2} \big\|\eta \big\|_{x_k}^{2} \right), 
    \end{equation}
    where $\tilde{L} > L$, $ L:=\max_{i=1,\ldots,m} L^i$, and $\psi_{x_k} \colon T_{x_k}\mathcal{M} \rightarrow \mathbb{R}$ is the multiobjective first-order approximation as follows
    \begin{equation}\label{psi}
    \psi_{x_k}(\eta) := \max_{i=1, \ldots, m} \Big( \left\langle\operatorname{grad} f_i\left(x_k\right), \eta\right\rangle_{x_k} + g_{i}(R_{x_k}(\eta))-g_{i}(x_k) \Big).
    \end{equation}
    For convenience, we define the subproblem objective function at $x \in \mathcal{M}$ as
    \begin{equation}\label{subproblem}
    p_{x}(\eta) := \psi_{x}(\eta) + \frac{\tilde{L}}{2} \big\|\eta \big\|_{x}^{2}.
    \end{equation}
    It follows immediately from \eqref{eq:pg_step} that $p_{x_k}(\eta_k)\leq p_{x_k}(0) = 0$.

    \begin{algorithm}[H]
      \caption{Riemannian Multiobjective Proximal Gradient Method (RMPGM)}
      \label{alg:VRPGD}
        \noindent Step 0. Let $x_0 \in \mathcal{M}$, $\tilde{L} > L$, and initialize $k \leftarrow 0$.\\
        Step 1. Compute the descent direction $\eta_k \in T_{x_k}\mathcal{M}$ by solving \eqref{eq:pg_step}.\\
        Step 2. Set $x_{k+1}=R_{x_k}(\eta_k)$, update $k \leftarrow k+1$, and return to Step 1.
    \end{algorithm}

\begin{remark}
    % \textbf{Practical implementation: Backtracking for $\tilde L$.}
Since the smoothness constant $L$ is usually unknown, Algorithm~\ref{alg:VRPGD} can be implemented with a standard backtracking strategy. Starting from an initial estimate $\tilde L_{-1}>0$, we increase $\tilde L_k$ geometrically until
    \[
        F_i(R_{x_k}(\eta_k)) \le F_i(x_k) - \frac{\tilde L_k - L}{2}\|\eta_k\|_{x_k}^2, \qquad \forall\, i = 1,\dots,m,
    \]
holds. By the $L$-retraction-smoothness of each $f_i$, the procedure terminates finitely once $\tilde L_k>L$. Moreover, the convergence analysis extends directly since Lemma~\ref{descent} matches the backtracking condition and ${\tilde L_k}$ remains uniformly bounded~\cite{huang_riemannian_2022,tanabe_proximal_2019}.
\end{remark}

\subsection{Global convergence analysis}
\label{sec: global convergenve}

% In the following, we adapt proof techniques from \cite{huang_riemannian_2022} to the multiobjective setting.
% First, to prove the global convergence, we need to generalize the $L$-smoothness to the Riemannian setting and define a notion of $L$-retraction-smooth.
% To establish global convergence, we extend the analysis in \cite{huang_riemannian_2022} to the multiobjective setting. In particular, we employ the notion of $L$-retraction-smoothness, which generalizes classical $L$-smoothness to Riemannian manifolds.
To establish global convergence, first, we introduce the notion of $L$-retraction-smoothness, which generalizes classical $L$-smoothness to Riemannian manifolds.
    \begin{definition}[\cite{huang_riemannian_2022}, L-retraction-smooth]
        A function $h: \mathcal{M} \rightarrow \mathbb{R}$ is called $L$-retraction-smooth with respect to a retraction $R$ in $\mathcal{N} \subseteq \mathcal{M}$ if for any $x \in \mathcal{N}$ and any $\mathcal{S}_x \subseteq \mathrm{T}_x \mathcal{M}$ such that $R_x\left(\mathcal{S}_x\right) \subseteq \mathcal{N}$, we have that
        \begin{equation}\label{retraction-smooth}
            h\left(R_x(\eta)\right) \leq h(x)+\langle\operatorname{grad} h(x), \eta\rangle_x+\frac{L}{2}\|\eta\|_x^2, \quad \forall \eta \in \mathcal{S}_x .    
        \end{equation}
    \end{definition}  

    The following assumptions will be used throughout the convergence analysis.
    \begin{assumption}\label{assume1}
        The function $F$ is bounded from below and the sublevel set $\omega_{x_0}  = \{x \in M | F(x) \preceq F(x_0)\}$ is compact.
    \end{assumption}
    \begin{assumption}\label{assume2}
        The function $f_i$ , for all $ i=1, \ldots, m$,  is $L^i$-retraction-smooth with respect to the retraction R in the sublevel set $\omega_{x_0}$ and $ L:=\max_{i=1,\ldots,m} L^i$.
    \end{assumption}

    \begin{lemma}[]\label{descent}
        Suppose Assumption \ref{assume2} holds. Then the sequence $\left\{x_k\right\}$ generated by Algorithm \ref{alg:VRPGD} satisfies
        \begin{equation}\label{Fi}
            F_i\left(x_k\right)-F_i\left(x_{k+1}\right) \geq \beta\left\|\eta_k\right\|_{x_k}^2, \quad i=1, \ldots, m,
        \end{equation}
        
        where $\beta=(\tilde{L}-L) / 2$.
    \end{lemma}
    \begin{proof}
        Considering the definition of $\eta_k$, it follows that $p_{x_k}(\eta_k) \leq 0$. This implies
        $$
        \left\langle\operatorname{grad} f_i\left(x_k\right), \eta_k\right\rangle_{x_k} + g_{i}(R_{x_k}(\eta_k))-g_{i}(x_k) \leq -\frac{\tilde{L}}{2}\left\|\eta_k\right\|_{x_k}^2, \quad i=1, \ldots, m,
        $$
        and combined with the $L$-retraction-smooth of $f_i$, we have
        $$
        \begin{aligned}
        F_i\left(x_k\right) & =f_i\left(x_k\right)+g_i\left(x_k\right) \\
        &\geq f_i\left(x_k\right)+\left\langle\operatorname{grad} f_i\left(x_k\right), \eta_k\right\rangle_{x_k}+g_i\left(R_{x_k}\left(\eta_k\right)\right)+\frac{\tilde{L}}{2}\left\|\eta_k\right\|_{x_k}^2 \\
        & \geq \frac{\tilde{L}-L}{2}\left\|\eta_k\right\|_{x_k}^2+f_i\left(R_{x_k}\left(\eta_k\right)\right)+g_i\left(R_{x_k}\left(\eta_k\right)\right)\\
        &=F_i\left(x_{k+1}\right)+\frac{\tilde{L}-L}{2}\left\|\eta_k\right\|_{x_k}^2,
        \end{aligned}
        $$
        which completes the proof.
    \end{proof}
\begin{lemma}[\cite{huang_riemannian_2022}, Lemma~2]
\label{lemma2}
Let $\xi$ be a continuous vector field on $\mathcal{M}$. 
For $x \in \mathcal{M}$ and $\eta_x \in T_x\mathcal{M}$, let 
$y = R_x(\eta_x)$. Then the operator
$(\mathrm{D}R_x(\eta_x)^*)^{-1}
:
T_x\mathcal{M}
\to
T_y\mathcal{M}$
is well defined and coincides with the inverse vector transport 
by differentiated retraction, denoted by 
$\mathcal{T}_{\eta_x}^{-\sharp}$ in~\cite{huang_riemannian_2022}. Moreover,
\[
\lim_{\eta_x \to 0}
\left\|
\xi_y
-
(\mathrm{D}R_x(\eta_x)^*)^{-1}\xi_x
\right\|_y
=
0.
\]
\end{lemma}

\begin{lemma}\label{lemma:sq_sum}
    Let $\{\eta_k\}$ be generated by Algorithm \ref{alg:VRPGD} and suppose that $\{F_i(x_k)\}$ is bounded from below for all $i = 1,\ldots,m$. Then, it follows that 
    \begin{equation}
        \sum_{k=0}^{\infty}\left\|\eta_k\right\|_{x_k}^2<\infty.
    \end{equation}
\end{lemma}
\begin{proof}
    From Lemma \ref{descent}, it follows that
    \begin{equation*}
        F_i\left(x_{k+1}\right) \leq F_i\left(x_k\right) - \beta\left\|\eta_k\right\|_{x_k}^2, \quad i=1, \ldots, m.
    \end{equation*}
    Since each $\{F_i(x_k)\}$ is bounded from below, there exists a constant $\bar{F}_i$ such that $F_i(x_k) \geq \bar{F}_i$ for all $k$. Summing the above inequality from $k = 0$ to $K$, we obtain
    \begin{equation*}
        F_i\left(x_{K+1}\right) \leq F_i\left(x_0\right) - \beta\sum_{k=0}^{K}\left\|\eta_k\right\|_{x_k}^2, \quad i=1, \ldots, m.
    \end{equation*}
    Rearranging terms gives
    \begin{equation*}
        \sum_{k=0}^{K}\left\|\eta_k\right\|_{x_k}^2 \leq \frac{1}{\beta}(F_i\left(x_0\right) - F_i(x_{K+1})) \leq \frac{1}{\beta}(F_i\left(x_0\right) - \bar{F}_i).
    \end{equation*}
    Taking the limit as $K \to \infty$, we conclude that $\sum_{k=0}^{\infty}\left\|\eta_k\right\|_{x_k}^2<\infty$.
\end{proof}

    Now we can give a global convergence analysis of Algorithm \ref{alg:VRPGD}.

\begin{theorem}%\cite{huang_riemannian_2022}
\label{thm:convergence}
    Under Assumptions \ref{assume0}, \ref{assume1}, and \ref{assume2}, the following statements hold:

    1. If $\eta_k=0$, then $x_k$ is a Pareto stationary point. 
    
    2. The sequence $\left\{x_k\right\}$ has at least one accumulation point. 
    
    3. Let $x_*$ be any accumulation point of the sequence $\left\{x_k\right\}$. Then $x_*$ is a Pareto stationary point. Furthermore, Algorithm \ref{alg:VRPGD} returns $x_k$ satisfying $\left\|\eta_k\right\|_{x_k} \leq \epsilon$ in at most $ \min_{i= 1, \ldots, m} \{\left(F_i\left(x_0\right)-F_i\left(x_*\right)\right) /\left(\beta \epsilon^2\right)$\}  iterations.
\end{theorem}
\begin{proof}
    Recall that for every $k$ the vector $\eta_k$ minimizes
    \[
        p_{x_k}(\eta)=\psi_{x_k}(\eta)+\frac{\tilde L}{2}\|\eta\|_{x_k}^2,
    \]
    and $p_{x_k}(0)=0$.

    1. Assume that $\eta_k=0$. Since $0$ is a minimizer of $p_{x_k}$, for every $d\in T_{x_k}\mathcal M$ and every sufficiently small $t>0$ we have
    \[
        0=p_{x_k}(0)\le p_{x_k}(td)
        =\max_{i=1,\ldots,m}\Big(\big\langle \operatorname{grad} f_i(x_k),td\big\rangle_{x_k}
        +g_i(R_{x_k}(td))-g_i(x_k)\Big)
        +\frac{\tilde L}{2}t^2\|d\|_{x_k}^2.
    \]
    Dividing by $t>0$, taking the limit superior as $t\downarrow 0$, and using the Clarke sum rule \eqref{eq:clarke-sum-rule}~\cite{hosseini_line_2018}, we obtain
    \[
        0\le \max_{i=1,\ldots,m}\Big(\big\langle \operatorname{grad} f_i(x_k),d\big\rangle_{x_k}+g_i^{\circ}(x_k;d)\Big)
        =\max_{i=1,\ldots,m}F_i^{\circ}(x_k;d).
    \]
    % \red{(Note: $g_i^\circ(x_k; d)$ is defined as $\limsup_{t\downarrow 0, y\to x_k} (g_i(R_y(t d_y)) - g_i(y))/t$, where $d_y$ denotes a smooth extension of $d$ around $x_k$; the limit superior of both sides preserves the inequality.)}
    % Since $d\in T_{x_k}\mathcal M$ is arbitrary, \eqref{stationary} holds and therefore $x_k$ is a Pareto stationary point.
    Here, $g_i^\circ(x_k; d)$ is defined as $\limsup_{t\downarrow 0, y\to x_k} (g_i(R_y(t d_y)) - g_i(y))/t$, where $d_y$ denotes a smooth extension of $d$ in a neighborhood of $x_k$. Taking the limit superior on both sides preserves the inequality. Given that $d\in T_{x_k}\mathcal M$ is arbitrary, condition \eqref{stationary} holds, which establishes $x_k$ as a Pareto stationary point.
    
    2. By Lemma \ref{descent}, $F_i(x_{k+1})\le F_i(x_k)$ for every $i=1,\ldots,m$. Hence
    \[
        F(x_k)\preceq F(x_0)\qquad \text{for all }k,
    \]
    and consequently $x_k\in \omega_{x_0}$ for all $k$. Assumption \ref{assume1} states that $\omega_{x_0}$ is compact, so the sequence $\{x_k\}$ has at least one accumulation point.

    3. Lemma~\ref{descent} implies that, for every $i$,
    \[
    F_i(x_{k+1})
    \le
    F_i(x_k)-\beta\|\eta_k\|_{x_k}^2,
    \qquad
    \beta=\frac{\tilde L-L}{2}>0.
    \]
    % Therefore, the update $x_{k+1}=R_{x_k}(\eta_k)$ satisfies the sufficient descent condition without backtracking. In particular, $x_{k+1}$ is the retraction image of $\eta_k$.

    Let $x_*$ be any accumulation point of $\{x_k\}$ and choose a subsequence $\{x_{k_j}\}$ such that
    \[
        x_{k_j}\to x_* \qquad \text{as } j\to\infty.
    \]
    % We prove that $x_*$ is Pareto stationary using a multiobjective KKT analysis inspired by \cite{huang_riemannian_2022}.
    
    The subproblem $p_{x_k}(\eta)=\max_{i=1,\dots,m}\varphi_i(x_k,\eta)+\frac{\tilde L}{2}\|\eta\|_{x_k}^2$ involves the maximum of  $\varphi_i(x_k,\eta):=\langle\operatorname{grad} f_i(x_k),\eta\rangle_{x_k}+g_i(R_{x_k}(\eta))-g_i(x_k)$.  By the Clarke subdifferential inclusion for the pointwise maximum of locally Lipschitz functions~\cite[Proposition~2.3.12]{clarke1990optimization}, together with the Riemannian chain rule for the subdifferential of $g_i\circ R_{x_k}$~\cite{hosseini2011generalized}, the necessary optimality condition $0\in\partial p_{x_{k_j}}( \eta_{k_j})$ yields the existence of multipliers $\lambda^{(k_j)} \in \Delta_m$   and subgradients $\zeta_{i, x_{k_j+1}} \in \partial g_i(x_{k_j+1})$ such that
    \begin{equation}\label{eq:kkt-optimality}
        \sum_{i=1}^m \lambda_i^{(k_j)} \operatorname{grad} f_i(x_{k_j}) + \tilde{L} \eta_{k_j} + \sum_{i=1}^m \lambda_i^{(k_j)} \operatorname{D} R_{x_{k_j}}(\eta_{k_j})^* [\zeta_{i, x_{k_j+1}}] = 0,
    \end{equation}
    where $\operatorname{D} R_{x_{k_j}}(\eta_{k_j})^*\colon T_{x_{k_j+1}}\mathcal M\to T_{x_{k_j}}\mathcal M$ is the adjoint of the differential of the retraction, and complementary slackness holds: $\lambda_i^{(k_j)}>0$ only for those indices $i$ achieving the maximum in $\psi_{x_{k_j}}(\eta_{k_j})$.

    Now define
    \begin{equation}\label{eq:pareto-vector}
    \begin{aligned}
        v_{k_j} &:= \sum_{i=1}^m \lambda_i^{(k_j)} \operatorname{grad} f_i(x_{k_j+1}) + \sum_{i=1}^m \lambda_i^{(k_j)} \zeta_{i, x_{k_j+1}}.
    \end{aligned}
    \end{equation}
    By Assumption~\ref{assume0}, $\operatorname{grad} f_i(x_{k_j+1})+\zeta_{i,x_{k_j+1}}\in\partial F_i(x_{k_j+1})$, so $v_{k_j}$ is a convex combination of subgradients.
    
    From the KKT condition \eqref{eq:kkt-optimality} and the invertibility of $\operatorname{D} R_{x_{k_j}}(\eta_{k_j})^*$ for sufficiently large $j$, which follows from $\|\eta_{k_j}\|_{x_{k_j}}\to0$ and the fact that $\operatorname{D} R_{x_{k_j}}(0_{x_{k_j}})=\operatorname{id}$, guaranteeing that $\operatorname{D} R_{x_{k_j}}(\eta_{k_j})$ is a linear isomorphism for all large $j$ in a totally retractive neighborhood; see Lemma~\ref{lemma2}, then 
    \[
        \sum_{i=1}^m \lambda_i^{(k_j)} \zeta_{i, x_{k_j+1}} = - (\operatorname{D} R_{x_{k_j}}(\eta_{k_j})^{*})^{-1} \Bigl( \sum_{i=1}^m \lambda_i^{(k_j)} \operatorname{grad} f_i(x_{k_j}) + \tilde{L} \eta_{k_j} \Bigr).
    \]

    Substituting this expression into $v_{k_j}$ and applying Lemma~\ref{lemma2}, together with $\|\eta_{k_j}\|_{x_{k_j}}\to 0$ from Lemma~\ref{lemma:sq_sum}, yields
    \begin{equation}\label{eq:v-converges}
        \lim_{j\to\infty}\|v_{k_j}\|_{x_{k_j+1}}=0.
    \end{equation}
% \red{Here we use the continuity of $\operatorname{grad} f_i$ and the fact that
% $\operatorname{D}R_{x_{k_j}}(\eta_{k_j})^{*,-1}\to \mathrm{Id}$ as
% $\eta_{k_j}\to0$.}

    Since $\Delta_m$ is compact, there exists a further subsequence (still denoted $\{k_j\}$) and $\lambda^*\in\Delta_m$ such that $\lambda^{(k_j)}\to\lambda^*$. Set $I_*:=\{i:\lambda_i^*>0\}$; then $\sum_{i\in I_*}\lambda_i^*=1$. Writing $s_{i,j}:=\operatorname{grad} f_i(x_{k_j+1})+\zeta_{i,x_{k_j+1}}$, we decompose
    \[
        v_{k_j}= \sum_{i\in I_*}\lambda_i^* s_{i,j} \;+\; \sum_{i\in I_*}(\lambda_i^{(k_j)}-\lambda_i^*)s_{i,j} \;+\; \sum_{i\notin I_*}\lambda_i^{(k_j)} s_{i,j}.
    \]

    Under Assumption~\ref{assume0}, each $F_i$ is locally Lipschitz on a neighborhood of $\omega_{x_0}$, hence the subdifferentials $\partial F_i(x_{k_j+1})$ are uniformly bounded~\cite[Proposition~2.5]{hosseini_line_2018}. Since $\lambda_i^{(k_j)}\to\lambda_i^*$ and $\lambda_i^{(k_j)}\to0$ for $i\notin I_*$, the last two sums converge to zero, yielding
    \begin{equation}\label{eq:fixed-lambda}
        \lim_{j\to\infty} \Bigl\|\sum_{i\in I_*} \lambda_i^* s_{i,j}\Bigr\|_{x_{k_j+1}} = 0.
    \end{equation}

    Given an arbitrary direction $d\in T_{x_*}\mathcal M$, we map it to a tangent vector at $x_{k_j+1}$ utilizing the inverse adjoint of the retraction differential. Specifically, we define $d_{k_j} := (\operatorname{D} R_{x_{k_j}}(\eta_{k_j}))^{*,-1}\bigl[\mathcal{T}_{x_*, x_{k_j}}(d)\bigr]$, where $\mathcal{T}_{x_*, x_{k_j}}\colon T_{x_*}\mathcal M\to T_{x_{k_j}}\mathcal M$ denotes the vector transport induced by the differentiated retraction~\cite{huang_riemannian_2022}. As $x_{k_j}\to x_*$, Lemma~\ref{lemma2} ensures that $\mathcal{T}_{x_*, x_{k_j}}(d)\to d$ in $T\mathcal M$, thereby keeping $\|d_{k_j}\|_{x_{k_j+1}}$ uniformly bounded. Consequently, noting that $s_{i,j}\in\partial F_i(x_{k_j+1})$, we have
    \[
        F_i^\circ(x_{k_j+1}; d_{k_j}) \ge \langle s_{i,j}, d_{k_j}\rangle_{x_{k_j+1}}.
    \]
    Multiplying by $\lambda_i^*$ and summing over $I_*$ yields $\sum_{i\in I_*}\lambda_i^*F_i^\circ(x_{k_j+1};d_{k_j})\ge\langle\sum_{i\in I_*}\lambda_i^* s_{i,j}, d_{k_j}\rangle$. Letting $j\to\infty$ and applying~\eqref{eq:fixed-lambda}, we obtain
    \[
        \liminf_{j\to\infty} \sum_{i\in I_*} \lambda_i^* F_i^\circ(x_{k_j+1}; d_{k_j}) \ge 0.
    \]

    According to~\cite[Theorem~2.2]{hosseini_line_2018}, the Clarke directional derivative $F_i^\circ(\cdot;\cdot)$ is upper semicontinuous on the tangent bundle.
    Since $d_{k_j}$ is constructed via smooth vector transport and converges in $T\mathcal M$ to $d$, the upper semicontinuity of $F_i^\circ$ as a function on the tangent bundle gives
    \[
        \sum_{i\in I_*} \lambda_i^* F_i^\circ(x_*; d) \;\ge\; \limsup_{j\to\infty} \sum_{i\in I_*} \lambda_i^* F_i^\circ(x_{k_j+1}; d_{k_j}) \;\ge\; 0.
    \]

    Since $\sum_{i\in I_*}\lambda_i^*=1$ with $\lambda_i^*\ge0$, the maximum is at least their convex combination, so
    \[
        \max_{i=1,\dots,m} F_i^\circ(x_*; d) \ge \sum_{i\in I_*} \lambda_i^* F_i^\circ(x_*; d) \ge 0.
    \]
    This holds for every $d\in T_{x_*}\mathcal M$, so \eqref{stationary} is satisfied, $x_*$ is a Pareto stationary point.

    Finally, let $\epsilon>0$ and suppose that $\|\eta_k\|_{x_k}>\epsilon$ for $k=0,1,\ldots,K-1$. Then Lemma \ref{descent} implies that, for every $i=1,\ldots,m$,
    \[
        F_i(x_0)-F_i(x_K)
        =\sum_{k=0}^{K-1}\big(F_i(x_k)-F_i(x_{k+1})\big)
        \ge \beta \sum_{k=0}^{K-1}\|\eta_k\|_{x_k}^2
        >K\beta \epsilon^2.
    \]
    Let $\{x_{k_j}\}$ be a subsequence converging to $x_*$. 
    By monotonicity and lower semicontinuity, we have
    % Since $\{F_i(x_k)\}$ is monotonically non-increasing, for any fixed $K$ and any $k_j > K$, we have $F_i(x_K) \ge F_i(x_{k_j})$. Taking the limit inferior as $j\to\infty$ and using the lower semicontinuity of $F_i$,
    \[
        F_i(x_K) \ge \liminf_{j\to\infty} F_i(x_{k_j}) \ge F_i(x_*).
    \]
    Therefore,
    \[
        F_i(x_0)-F_i(x_*)\ge F_i(x_0)-F_i(x_K)>K\beta\epsilon^2,
    \]
    that is,
    \[
        K<\frac{F_i(x_0)-F_i(x_*)}{\beta\epsilon^2},\qquad i=1,\ldots,m.
    \]
    Hence Algorithm \ref{alg:VRPGD} returns an iterate $x_k$ satisfying $\|\eta_k\|_{x_k}\le \epsilon$ in at most
    \[
        \min_{i=1,\ldots,m}\left\{\frac{F_i(x_0)-F_i(x_*)}{\beta\epsilon^2}\right\}
    \]
    iterations.
\end{proof}

\begin{remark}
Part~3 establishes Pareto stationarity directly from the KKT conditions of the tangent-space subproblem, combined with vector transport and the upper semicontinuity of the Clarke generalized directional derivative. This avoids requiring a closed graph property for the Pareto subdifferential mapping, which does not follow automatically from the componentwise subdifferentials.
\end{remark}

\subsection{Convergence rate analysis}
\subsubsection{Global convergence rate analysis using retraction conversity}

For convex problems in Euclidean spaces, the multiobjective proximal gradient method achieves an $O(1/k)$ convergence rate~\cite{tanabe_convergence_2023}. To adopt the same result to the Riemannian setting, we adopt the notion of retraction-convexity introduced in~\cite{huang_riemannian_2022}.

    \begin{definition}{\cite[Definition 2]{huang_riemannian_2022}\label{def:retraction-convex}}
        A function $h: \mathcal{M} \rightarrow \mathbb{R}$ is called retraction-convex with respect to a retraction $R$ in $\mathcal{N} \subseteq \mathcal{M}$ if for any $x \in \mathcal{N}$ and any $\mathcal{S}_x \subseteq \mathrm{T}_x \mathcal{M}$ such that $R_x\left(\mathcal{S}_x\right) \subseteq \mathcal{N}$,
        the composition $q_x=h \circ R_x$ satisfies: for every $\xi \in \mathcal{S}_x$, there exists $\zeta \in \mathrm{T}_x \mathcal{M}$ such that
        \begin{equation}
            q_x(\eta) \geq q_x(\xi)+\langle\zeta, \eta-\xi\rangle_x \quad \forall \eta \in \mathcal{S}_x.
        \end{equation}
        In other words, $q_x$ admits a supporting tangent vector (subgradient) at every point of $\mathcal{S}_x$.
    \end{definition}

    Note that $\zeta=\operatorname{grad} q_x(\xi)$ if $h$ is differentiable; otherwise, $\zeta$ is any Riemannian subgradient of $q_x$ at $\xi$.

\begin{remark}

Retraction-convexity generalizes geodesic convexity by replacing the exponential map with a general retraction $R$. This allows the convergence analysis to accommodate practical retractions that are computationally cheaper than the exponential map. For conditions under which geodesically convex functions inherently satisfy retraction-convexity via second-order retractions, we refer the reader to~\cite{huang_riemannian_2022}.

\end{remark}

The following two assumptions are used in the convergence rate analysis.
    \begin{assumption}\label{assume3}
        % \red{For all $i$, there exists an open set $\Omega \supseteq \omega_{x_0}$ such that the function $f_i$ is $L$-retraction-smooth and retraction-convex with respect to the retraction $R$ in $\Omega$. The function $g_i$ is retraction-convex with respect to the retraction $R$ in $\Omega$.}
        There exists an open neighborhood $\Omega \supseteq \omega_{x_0}$ such that, for all $i=1,\dots,m$, the function $f_i$ is $L$-retraction-smooth, and both $f_i$ and $g_i$ are retraction-convex with respect to the retraction $R$ in $\Omega$.
    \end{assumption}

    \begin{assumption}\cite[Assumption 4]{huang_riemannian_2022}\label{assume4}
        For any $x, y, z \in \Omega$, there exists a constant $\kappa_{\Omega}$ such that
        \begin{equation}\label{ineq:ass4}
            \left|\left\|\xi_x-\eta_x\right\|_x^2-\left\|\zeta_y\right\|_y^2\right| \leq \kappa_{\Omega}\left\|\eta_x\right\|_x^2
        \end{equation}
        where $\eta_x=R_x^{-1}(y), \xi_x=R_x^{-1}(z), \zeta_y=R_y^{-1}(z)$, and $\kappa_{\Omega}$ is a constant.
    \end{assumption}

Assumption~\ref{assume4} holds for the exponential map on geodesically convex regions where the inverse exponential map is well defined; see~\cite[Lemma~2.4]{daniilidis2018self}. 
For general retractions, the estimate  generally holds only in a local neighborhood.

The following lemma is fundamental to the subsequent convergence rate analysis.
    \begin{lemma}
    \label{lemma4.2}
        Let $\eta_x$ be the unique minimizer of $p_{x}(\eta)$ defined in \eqref{subproblem} such that $p_x(0) \geq p_x(\eta_x)$. Suppose Assumption \ref{assume3} holds, and $x, z=R_x(\eta_x)$ are in $\Omega$. Then for any $\xi_x \in \mathrm{T}_x \mathcal{M}$ such that $y:=R_x(\xi_x) \in \Omega$, there exist multipliers $\lambda_i(x) \geq 0$ with $\sum_{i=1}^m \lambda_i(x) = 1$ such that
        $$
        \sum_{i=1}^m \lambda_i(x) F_i(z) \leq \sum_{i=1}^m \lambda_i(x) F_i(y)+\frac{\tilde{L}}{2}\left(\left\|\xi_x\right\|_x^2-\left\|\xi_x-\eta_x\right\|_x^2\right).
        $$
    \end{lemma}
    \begin{proof}
        Using the optimality of $\eta_x$ for subproblem \eqref{subproblem}, the subdifferential rule for pointwise maxima of convex functions~\cite{clarke1990optimization}, and the Riemannian chain rule~\cite{hosseini2011generalized} on the tangent space $T_x\mathcal M$,there exist KKT multipliers $\lambda_i(x) \geq 0$ with $\sum_{i=1}^m \lambda_i(x) = 1$ and subgradients $\zeta_{i, z} \in \partial g_i(z)$ (where $z=R_x(\eta_x)$) such that:
        \begin{equation}\label{eq:kkt-lemma42}
              \sum_{i=1}^m \lambda_i(x) \operatorname{grad} f_i(x) + \tilde{L} \eta_x + \sum_{i=1}^m \lambda_i(x) \operatorname{D} R_x(\eta_x)^* [\zeta_{i, z}] = 0,
        \end{equation}
        where $\operatorname{D} R_x(\eta_x)^* \colon T_{z}\mathcal M \to T_{x}\mathcal M$. Since each $g_i$ is retraction-convex and $y=R_x(\xi_x), z \in \Omega$, convexity of $g_i \circ R_x$ at $\eta_x$ applied to the direction $\xi_x - \eta_x$ yields
        \begin{equation}\label{eq:g-convex-lemma42}
            g_i(y) - g_i(z) \ge \langle \operatorname{D} R_x(\eta_x)^* [\zeta_{i, z}],\; \xi_x - \eta_x \rangle_x .
        \end{equation}
        Rearranging \eqref{eq:g-convex-lemma42} gives $g_i(z) \le g_i(y) + \langle \operatorname{D} R_x(\eta_x)^* [\zeta_{i, z}],\; \eta_x - \xi_x \rangle_x$.
        Multiplying by $\lambda_i(x)$ and summing over $i$,
        \begin{equation}\label{ineq:g}
            \sum_{i=1}^m \lambda_i(x) g_i(z) \le \sum_{i=1}^m \lambda_i(x) g_i(y) + \Bigl\langle \sum_{i=1}^m \lambda_i(x) \operatorname{D} R_x(\eta_x)^* [\zeta_{i, z}],\; \eta_x - \xi_x \Bigr\rangle_x .
        \end{equation}
        Substituting the KKT condition \eqref{eq:kkt-lemma42}, namely $\sum_{i=1}^m \lambda_i(x) \operatorname{D} R_x(\eta_x)^* [\zeta_{i, z}] = -(\sum_{i=1}^m \lambda_i(x) \operatorname{grad} f_i(x) + \tilde L \eta_x)$, into \eqref{ineq:g} yields
        \begin{equation}\label{ineq:g2}
            \sum_{i=1}^m \lambda_i(x) g_i(z) \le \sum_{i=1}^m \lambda_i(x) g_i(y) + \Bigl\langle -\sum_{i=1}^m \lambda_i(x) \operatorname{grad} f_i(x) - \tilde L \eta_x,\; \eta_x - \xi_x \Bigr\rangle_x .
        \end{equation}
        Now, using the Assumption \ref{assume2} and Assumption \ref{assume3}, we can bound $\sum_{i=1}^m \lambda_i(x) F_i(z)$ as follows
        \begin{equation}\label{eq:expanded_lemma}
        \begin{aligned}
        \sum_{i=1}^m \lambda_i(x) F_i(z) & = \sum_{i=1}^m \lambda_i(x) (f_i(z) + g_i(z)) \\
        & \le \sum_{i=1}^m \lambda_i(x) \Bigl( f_i(x) + \langle \operatorname{grad} f_i(x), \eta_x \rangle_x + \frac{L}{2} \|\eta_x\|_x^2 \Bigr) + \sum_{i=1}^m \lambda_i(x) g_i(z) \\
        & \le \sum_{i=1}^m \lambda_i(x) f_i(x) + \Bigl\langle \sum_{i=1}^m \lambda_i(x) \operatorname{grad} f_i(x), \eta_x \Bigr\rangle_x + \frac{\tilde L}{2} \|\eta_x\|_x^2 \\
        & \quad + \sum_{i=1}^m \lambda_i(x) g_i(y) + \Bigl\langle -\sum_{i=1}^m \lambda_i(x) \operatorname{grad} f_i(x) - \tilde L \eta_x,\; \eta_x - \xi_x \Bigr\rangle_x \\
        & = \sum_{i=1}^m \lambda_i(x) f_i(x) + \Bigl\langle \sum_{i=1}^m \lambda_i(x) \operatorname{grad} f_i(x), \xi_x \Bigr\rangle_x + \sum_{i=1}^m \lambda_i(x) g_i(y) \\
        & \quad + \tilde L \langle \eta_x, \xi_x - \eta_x \rangle_x + \frac{\tilde L}{2} \|\eta_x\|_x^2 \\
        & \le \sum_{i=1}^m \lambda_i(x) (f_i(y) + g_i(y)) + \frac{\tilde L}{2}\bigl(2\langle \eta_x, \xi_x \rangle_x - \|\eta_x\|_x^2 \bigr) \\
        & = \sum_{i=1}^m \lambda_i(x) F_i(y) + \frac{\tilde L}{2} \bigl( \|\xi_x\|_x^2 - \|\xi_x - \eta_x\|_x^2 \bigr).
        \end{aligned}
        \end{equation}
        In the last inequality we used the retraction-convexity of $f_i$: $f_i(x) + \langle \operatorname{grad} f_i(x), \xi_x \rangle_x \le f_i(R_x(\xi_x)) = f_i(y)$. The final equality uses the polarization identity $2\langle \eta_x, \xi_x \rangle_x - \|\eta_x\|_x^2 = \|\xi_x\|_x^2 - \|\xi_x - \eta_x\|_x^2$.
    \end{proof}

To establish a global convergence rate for multiobjective optimization, we adopt the merit function from~\cite{tanabe2024new} as a stationarity measure. Given a compact neighborhood $\Omega$, define
\begin{equation}
    u_0(x)
    :=
    \sup_{y \in \Omega}
    \min_{i=1,\dots,m}
    \{F_i(x)-F_i(y)\}.
\end{equation}
Clearly, $u_0(x)\ge0$ for all $x\in\Omega$, and $u_0(x)=0$ if and only if $x$ is a weakly Pareto optimal point in $\Omega$.
For the complexity analysis, we further consider the averaged quantity
\begin{equation}
    \bar u_k
    :=
    \sup_{y\in\Omega}
    \frac1k
    \sum_{s=0}^{k-1}
    \min_i
    \{F_i(x_{s+1})-F_i(y)\},
\end{equation}
which can be viewed as an ergodic counterpart of the merit function $u_0$.

\begin{theorem}\label{thm:g_rate}
Suppose Assumptions \ref{assume0}, \ref{assume1}, \ref{assume3}, and \ref{assume4} hold. Let $\Omega$ be the compact neighborhood defined in the assumptions. Let
$
D:=\sup_{y\in\Omega}\|R_{x_0}^{-1}(y)\|_{x_0},
$
and 
$
F_i^*:=\inf_{x\in\Omega}F_i(x).
$
Then, the sequence $\{x_k\}$ generated by Algorithm \ref{alg:VRPGD} satisfies
\begin{equation}\label{eq:thmrate}
    \bar u_k
    \le
    \frac1k
    \left(
        \frac{\tilde L}{2}D^2
        +
        \frac{\tilde L\kappa_\Omega}{2\beta}
        \min_{i=1,\dots,m}
        (F_i(x_0)-F_i^*)
    \right),
\end{equation}
where $\kappa_\Omega$ and $\beta$ are defined in Assumption \ref{assume4} and Lemma \ref{descent}, respectively.
\end{theorem}

\begin{proof}
Let $\eta_s$ be the unique minimizer of the subproblem $p_{x_s}(\eta)$ at iteration $s$. By Lemma \ref{lemma4.2}, there exists a KKT multiplier vector $\lambda(x_s)\in\Delta_m$ such that for any $y\in\Omega$,
\begin{equation}\label{eq:rate_step1}
    \sum_{i=1}^m
    \lambda_i(x_s)
    \bigl(F_i(x_{s+1})-F_i(y)\bigr)
    \le
    \frac{\tilde L}{2}
    \left(
        \|R_{x_s}^{-1}(y)\|_{x_s}^2
        -
        \|R_{x_s}^{-1}(y)-\eta_s\|_{x_s}^2
    \right).
\end{equation}
Since $\lambda(x_s)\in\Delta_m$, it follows that
\begin{equation}\label{eq:rate_step2}
    \min_{i=1,\dots,m}
    \{F_i(x_{s+1})-F_i(y)\}
    \le
    \sum_{i=1}^m
    \lambda_i(x_s)
    \bigl(F_i(x_{s+1})-F_i(y)\bigr).
\end{equation}
Combining \eqref{eq:rate_step1} and \eqref{eq:rate_step2} we have
\begin{equation}\label{eq:rate_step3}
    \min_{i=1,\dots,m}
    \{F_i(x_{s+1})-F_i(y)\}
    \le
    \frac{\tilde L}{2}
    \left(
        \|R_{x_s}^{-1}(y)\|_{x_s}^2
        -
        \|R_{x_s}^{-1}(y)-\eta_s\|_{x_s}^2
    \right).
\end{equation}
Applying Assumption \ref{assume4} with $x=x_s$, $\tilde y=x_{s+1}=R_{x_s}(\eta_s)$, and $z=y$, we obtain
\[
-
\|R_{x_s}^{-1}(y)-\eta_s\|_{x_s}^2
\le
-
\|R_{x_{s+1}}^{-1}(y)\|_{x_{s+1}}^2
+
\kappa_\Omega\|\eta_s\|_{x_s}^2.
\]
Substituting this into \eqref{eq:rate_step3} gives
\begin{equation}\label{eq:rate_step4}
\begin{aligned}
    \min_{i=1,\dots,m}
    \{F_i(x_{s+1})-F_i(y)\}
    \le
    \frac{\tilde L}{2}
    \Big(
        \|R_{x_s}^{-1}(y)\|_{x_s}^2
        -
        \|R_{x_{s+1}}^{-1}(y)\|_{x_{s+1}}^2
    \Big)
    +
    \frac{\tilde L\kappa_\Omega}{2}
    \|\eta_s\|_{x_s}^2.
\end{aligned}
\end{equation}
Summing \eqref{eq:rate_step4} from $s = 0$ to $k-1$ yields
    \begin{equation}\label{eq:rate_sum}
    \begin{aligned}
        \sum_{s=0}^{k-1} \min_{i=1,\dots,m} \left\{ F_i(x_{s+1}) - F_i(y) \right\} &\leq \frac{\tilde{L}}{2} \left( \|R_{x_0}^{-1}(y)\|_{x_0}^2 - \|R_{x_k}^{-1}(y)\|_{x_k}^2 \right) + \frac{\tilde{L}\kappa_\Omega}{2} \sum_{s=0}^{k-1} \|\eta_s\|_{x_s}^2 \\
        &\leq \frac{\tilde{L}}{2}\|R_{x_0}^{-1}(y)\|_{x_0}^2 + \frac{\tilde{L}\kappa_\Omega}{2} \sum_{s=0}^{k-1} \|\eta_s\|_{x_s}^2,
    \end{aligned}
    \end{equation}
    where   $- \frac{\tilde{L}}{2} \|R_{x_k}^{-1}(y)\|_{x_k}^2 \leq 0$.

Since \eqref{eq:rate_sum} holds for every $y\in\Omega$, taking the supremum over $y\in\Omega$ on both sides yields
\begin{equation}\label{eq:rate_sup}
\sup_{y\in\Omega}
\sum_{s=0}^{k-1}
\min_i
\{F_i(x_{s+1})-F_i(y)\}
\le
\frac{\tilde L}{2}D^2
+
\frac{\tilde L\kappa_\Omega}{2}
\sum_{s=0}^{k-1}
\|\eta_s\|_{x_s}^2.
\end{equation}

By the definition of $\bar u_k$,
\begin{equation}\label{eq:rate_lhs_u}
    \bar u_k
    \le
    \frac1k
    \left(
        \frac{\tilde L}{2}D^2
        +
        \frac{\tilde L\kappa_\Omega}{2}
        \sum_{s=0}^{k-1}
        \|\eta_s\|_{x_s}^2
    \right).
\end{equation}
    Also from Lemma \ref{descent}, $F_i(x_s) - F_i(x_{s+1}) \geq \beta \|\eta_s\|_{x_s}^2$ for every $i$. Summing over $s$ and taking the minimum over $j$ gives
    \begin{equation}\label{eq:rate_rhs}
        \sum_{s=0}^{k-1} \|\eta_s\|_{x_s}^2 \leq \frac{1}{\beta} \min_{j=1,\dots,m} \left( F_j(x_0) - F_j(x_k) \right) \leq \frac{1}{\beta} \min_{j=1,\dots,m} \left( F_j(x_0) - F_j(x_*) \right).
    \end{equation}
    Substituting \eqref{eq:rate_rhs} into \eqref{eq:rate_lhs_u}  yields \eqref{eq:thmrate}.
\end{proof}

In addition to the ergodic result, we can also establish a non-ergodic convergence rate.

    \begin{theorem}\label{thm:g_rate2}
Suppose Assumptions \ref{assume0}, \ref{assume1}, \ref{assume3}, and \ref{assume4} hold. Let $x_* \in \Omega$ be any accumulation point of the sequence $\{x_k\}$ generated by Algorithm \ref{alg:VRPGD}. Then, for all $k \geq 1$,
\begin{equation}\label{eq:thm2}
\min_{i=1,\dots,m} \left( F_i(x_k) - F_i(x_*) \right)
\leq
\frac{1}{k}
\left(
\frac{\tilde{L}}{2}\|R_{x_0}^{-1}(x_*)\|_{x_0}^2
+
\frac{\tilde{L}\kappa_\Omega}{2\beta}
\min_{j=1,\dots,m}
\left( F_j(x_0)-F_j(x_*) \right)
\right),
\end{equation}
where $\kappa_\Omega$ and $\beta$ are defined in Assumption \ref{assume4} and Lemma \ref{descent}, respectively.
\end{theorem}

\begin{proof}
Following the proof of Theorem~\ref{thm:g_rate}, for any accumulation point $x_*$ and all $s\ge0$, we obtain
\begin{equation}\label{eq:rate_step22}
\min_{i=1,\dots,m}
\bigl(F_i(x_{s+1})-F_i(x_*)\bigr)
\leq
\frac{\tilde L}{2}
\left(
\|R_{x_s}^{-1}(x_*)\|_{x_s}^2
-
\|R_{x_s}^{-1}(x_*)-\eta_s\|_{x_s}^2
\right).
\end{equation}
Applying Assumption~\ref{assume4} with $x=x_s$, $\tilde y=x_{s+1}=R_{x_s}(\eta_s)$, and $z=x_*$, and summing from $s=0$ to $k-1$, yields
\begin{equation}\label{eq:rate_sum2}
\begin{aligned}
\sum_{s=0}^{k-1}
\min_i
\bigl(F_i(x_{s+1})-F_i(x_*)\bigr)
\leq
\frac{\tilde L}{2}
\|R_{x_0}^{-1}(x_*)\|_{x_0}^2
+
\frac{\tilde L\kappa_\Omega}{2}
\sum_{s=0}^{k-1}
\|\eta_s\|_{x_s}^2.
\end{aligned}
\end{equation}

By Lemma~\ref{descent}, each sequence $\{F_i(x_k)\}$ is monotonically non-increasing and convergent. Since $x_*$ is an accumulation point and $F_i$ is continuous, then for all $i$, $F_i(x_k)\to F_i(x_*)$.
% \[
% F_i(x_k)\to F_i(x_*), \qquad i=1,\dots,m.
% \]
Hence,
\[
F_i(x_{s+1})\ge F_i(x_k)\ge F_i(x_*),
\qquad \forall s\le k-1,
\]
which implies
\begin{equation}\label{eq:rate_lhs}
\sum_{s=0}^{k-1}
\min_i
\bigl(F_i(x_{s+1})-F_i(x_*)\bigr)
\ge
k
\min_i
\bigl(F_i(x_k)-F_i(x_*)\bigr).
\end{equation}

Combining \eqref{eq:rate_sum2}, \eqref{eq:rate_lhs}, and the estimate on $\sum_{s=0}^{k-1}\|\eta_s\|_{x_s}^2$ from the proof of Theorem~\ref{thm:g_rate} yields the desired result.
\end{proof}

\subsubsection{Solving the Riemannian Proximal Mapping}
\label{sec: subproblem_computation}

The subproblem \eqref{eq:pg_step} requires minimizing $p_x(\eta) = \psi_x(\eta) + \frac{\tilde{L}}{2}\|\eta\|_x^2$ over $T_x\mathcal{M}$, where $\psi_x$ involves the composition $g_i \circ R_x$. For general retractions, this composition may be difficult to handle directly. Following the approach of~\cite{huang_riemannian_2022}, we develop an iterative solver that transfers the subproblem to successively updated tangent spaces.

\begin{assumption}\label{assump6}
    (i) The manifold $\mathcal{M}$ is an embedded submanifold of $\mathbb{R}^n$.
    (ii) Each $g_i$ is Lipschitz continuous with constant $L_g$ and convex in $\mathbb{R}^n$.
    (iii) Each $g_i$ is bounded from below on $\mathcal{M}$.
\end{assumption}

Fix $x \in \mathcal{M}$ and define the subproblem objective
\begin{equation}\label{eq:ell-def}
    \ell_x(\eta) := \max_{i=1, \ldots, m} \Big( \langle \operatorname{grad} f_i(x), \eta \rangle_x + g_i(R_x(\eta)) - g_i(x) \Big) + \frac{\tilde{L}}{2}\|\eta\|_x^2.
\end{equation}
Given a current estimate $\eta_k \in T_x\mathcal{M}$, we seek a correction $\tilde{\xi}_k \in T_x\mathcal{M}$ that decreases $\ell_x$. Set $y_k := R_x(\eta_k)$ and let $\xi_k := \mathcal{T}_{R_{\eta_k}} \tilde{\xi}_k \in T_{y_k}\mathcal{M}$ denote the transported correction, where $\mathcal{T}_{R_{\eta_k}}$ is the vector transport associated with $R$. Expanding $\ell_x(\eta_k + \tilde{\xi}_k)$ for any index $i$ attaining the maximum at $\eta_k$:
\begin{align*}
    \ell_x(\eta_k + \tilde{\xi}_k) &= \langle \operatorname{grad} f_i(x), \eta_k \rangle_x + \frac{\tilde{L}}{2}\|\eta_k\|_x^2 \\
    &\quad + \langle \operatorname{grad} f_i(x) + \tilde{L}\eta_k, \tilde{\xi}_k \rangle_x + \frac{\tilde{L}}{2}\|\tilde{\xi}_k\|_x^2 + g_i(R_x(\eta_k + \tilde{\xi}_k)) - g_i(x).
\end{align*}
Since $R$ is smooth, $R_x(\eta_k + \tilde{\xi}_k) = y_k + \xi_k + O(\|\xi_k\|^2)$. The Lipschitz continuity of $g_i$ (Assumption~\ref{assump6}(ii)) gives
$|g_i(y_k + \xi_k + O(\|\xi_k\|^2)) - g_i(y_k + \xi_k)| \le L_g \cdot O(\|\xi_k\|^2) = O(\|\xi_k\|^2)$.
Since $\|\mathcal{T}_{R_{\eta_k}}^{-1}\xi_k\|_x$ and $\|\xi_k\|_{y_k}$ are equivalent up to second-order terms on a compact embedded submanifold (by the smoothness of the metric and the transport), we obtain the local model
\begin{equation}\label{eq:transferred-subproblem}
    \ell_x(\eta_k + \tilde{\xi}_k) = \ell_x(\eta_k) + \tilde{\ell}_{y_k}(\xi_k) + O(\|\xi_k\|_{y_k}^2),
\end{equation}
where the transferred subproblem at $y_k \in \mathcal{M}$ is
\begin{equation}\label{sub_tran}
    \tilde{\ell}_{y_k}(\xi) := \max_{i \in I(\eta_k)} \Big( \langle \mathcal{T}_{R_{\eta_k}}^{-\sharp}(\operatorname{grad} f_i(x) + \tilde{L}\eta_k), \xi \rangle_{y_k} + g_i(y_k + \xi) - g_i(y_k) \Big) + \frac{\tilde{L}}{2}\|\xi\|_{y_k}^2,
\end{equation}
with $I(\eta_k) := \{j \in \{1,\ldots,m\} : \varphi_j(x,\eta_k) = \psi_x(\eta_k)\}$ the active index set, and $\mathcal{T}_{R_{\eta_k}}^{-\sharp}$ the adjoint of $\mathcal{T}_{R_{\eta_k}}^{-1}$.

The key observation is that \eqref{sub_tran} involves $g_i(y_k + \xi)$ evaluated in the ambient space, which preserves the convexity of $g_i$. For embedded submanifolds where $T_{y_k}\mathcal{M}$ is a linear subspace of $\mathbb{R}^n$, subproblem~\eqref{sub_tran} reduces to a finite-dimensional convex program that can be solved by standard methods.
% (e.g., SOCP solvers when $g_i = \mu\|\cdot\|_1$).

The descent direction $\xi_k^* := \arg\min_{\xi \in T_{y_k}\mathcal{M}} \tilde{\ell}_{y_k}(\xi)$ is then pulled back to $T_x\mathcal{M}$ via $\mathcal{T}_{R_{\eta_k}}^{-1}\xi_k^*$. An Armijo line search on the exact objective $\ell_x$ ensures sufficient decrease. The complete procedure is given in Algorithm~\ref{alg: mapping}.

\begin{algorithm}[ht]
  \caption{Iterative Solver for the Riemannian Proximal Mapping}
  \label{alg: mapping}
  \noindent \textbf{Input:} Base point $x \in \mathcal{M}$; parameter $\tilde{L} > L$; initial estimate $\eta_0 = 0_{x} \in T_x\mathcal{M}$; line search parameter $\sigma \in (0, 1)$.\\
  \textbf{for} $k = 0, 1, \ldots$ \textbf{do}\\
  \quad 1. Set $y_k = R_x(\eta_k)$.\\
  \quad 2. Compute $\xi_k^* = \arg\min_{\xi \in T_{y_k}\mathcal{M}} \tilde{\ell}_{y_k}(\xi)$ by solving \eqref{sub_tran}.\\
  \quad 3. Set step size $\alpha = 1$.\\
  \quad 4. \textbf{while} $\ell_x(\eta_k + \alpha \mathcal{T}_{R_{\eta_k}}^{-1}\xi_k^*) > \ell_x(\eta_k) - \sigma \alpha \|\xi_k^*\|_{y_k}^2$ \textbf{do} $\alpha \leftarrow \alpha / 2$. \\
  \quad 5. Update $\eta_{k+1} = \eta_k + \alpha \mathcal{T}_{R_{\eta_k}}^{-1}\xi_k^*$.\\
  \textbf{end for}
\end{algorithm}

\begin{remark}
    The Armijo line search in Step~4 evaluates the exact objective $\ell_x$, so that the sequence $\{\ell_x(\eta_k)\}$ generated by Algorithm~\ref{alg: mapping} is monotonically decreasing. The transferred subproblem~\eqref{sub_tran}, based on the ambient evaluation $g_i(y_k+\xi)$, serves only as a local model for the computation of the search direction.
    % This design maintains full consistency with the convergence theory of Sections~\ref{sec: MRPGD method}--\ref{sec: trust region}, which assumes the exact retraction-based subproblem.
\end{remark}

\section{Inexact version of the proposed method}
\label{sec: Inexact method}
The inexact method is proposed in Algorithm \ref{alg:inexact}.
Recall that the search direction at the $k$-th iteration minimizes the proximal mapping
    \begin{equation*}
    p_{x_{k}}(\eta) = \max_{i=1, \ldots, m} \Big( \left\langle\operatorname{grad} f_i\left(x_k\right), \eta\right\rangle_{x_k} + g_{i}(R_{x_k}(\eta))-g_{i}(x_k) \Big) + \frac{\tilde{L}}{2} \big\|\eta \big\|_{x_k}^{2}.
    \end{equation*}
Instead of solving this subproblem to exact optimality, which may be computationally expensive or admit no closed-form solution, we introduce a relaxed criterion that accepts an approximate step, as detailed in Algorithm~\ref{alg:inexact}. The proposed inexact framework ensures convergence even when the subproblems are solved only approximately.
    \begin{algorithm}[ht]
      \caption{Inexact Riemannian Multiobjective Proximal Gradient Method (Inexact-RMPGM)}
      \label{alg:inexact}
      \noindent Require: A constant $\tilde{L}>L$; an initial iterate $x_0$; \\
        1: for $k=0, 1, \ldots$, do\\
        2: \quad Find an inexact step $\hat{\eta}_k \in T_{x_k}\mathcal{M}$ such that $p_{x_k}(\hat{\eta}_k) \leq p_{x_k}(0)$ and there exists a subgradient $v_k \in \partial p_{x_k}(\hat{\eta}_k)$ satisfying:
        $$
        \|v_k\|_{x_k} \leq \varepsilon_k \|\hat{\eta}_k\|_{x_k},
        $$
        \quad \quad where $\varepsilon_k > 0$ is a forcing sequence with $\lim_{k\to\infty}\varepsilon_k = 0$; \\
        3: \quad $x_{k+1} = R_{x_k}(\hat{\eta}_k)$; \\
        4: end for
    \end{algorithm}

    % % \red{ Alternative practical remark removed as duplicate; see below.}
    % The condition $\|v_k\|_{x_k}\le\varepsilon_k\|\hat{\eta}_k\|_{x_k}$ is a theoretical device employed for convergence analysis. In practice, the full Clarke subdifferential $\partial p_{x_k}(\hat{\eta}_k)$ cannot be computed; practical implementations often employ surrogate stopping rules, such as bounding the objective gap $\|p_{x_k}(\hat{\eta}_k)-p_{x_k}(\eta^*_k)\|\le\varepsilon_k$ or monitoring the decrease of $p_{x_k}$ across inner iterations. Note that $\varepsilon_k$ governs the asymptotic accuracy: the rapid decay of the residual as $\|\hat{\eta}_k\|_{x_k}\to0$ and $\varepsilon_k\to0$ requires solving the subproblems to increasingly high precision near convergence. Such a trade-off is intrinsic to all inexact proximal methods~\cite{tanabe_proximal_2019,huang_inexact_2023}.
% The condition $\|v_k\|_{x_k}\le\varepsilon_k\|\hat{\eta}_k\|_{x_k}$ is mainly introduced for convergence analysis. 
In practice, the full Clarke subdifferential $\partial p_{x_k}(\hat{\eta}_k)$ is typically unavailable, and one instead uses practical stopping criteria, such as objective-gap conditions or monitoring the decrease of $p_{x_k}$ during the inner iterations. The parameter $\varepsilon_k$ controls the asymptotic accuracy: as $\|\hat{\eta}_k\|_{x_k}\to0$ and $\varepsilon_k\to0$, the subproblems must be solved more accurately near convergence, as is common in inexact proximal methods~\cite{tanabe_proximal_2019,huang_inexact_2023}.

To guarantee that all descent steps remain inside the region where the retraction-convexity assumption is valid,
we establish the following localization result.

\begin{lemma}
\label{lem:localization}
Suppose Assumptions~\ref{assume0}--\ref{assume3} hold.
Let $\omega_{x_0}$ be the compact level set in Assumption~\ref{assume1}$,$ and let $\Omega \supseteq \omega_{x_0}$ be the associated open neighborhood. Then there exists a radius $r_\Omega>0$ such that
\[
R_x\bigl(B(0_x,r_\Omega)\bigr)\subset \Omega,
\qquad
\forall x\in\omega_{x_0},
\]
where
$B(0_x,r_\Omega)
:=
\{\eta\in T_x\mathcal M:\|\eta\|_x\le r_\Omega\}.$
Moreover, there exists a constant $M>0$ such that for all $x\in\omega_{x_0}$ and all $\eta\in T_x\mathcal M$ satisfying  $\|\eta\|_x \le r_\Omega$,
\begin{equation}
\label{eq:localization_bound}
\max_{i=1,\dots,m}
\Bigl(
\langle \operatorname{grad} f_i(x),\eta\rangle_x
+
g_i(R_x(\eta))
-
g_i(x)
\Bigr)
\ge
-M\|\eta\|_x.
\end{equation}
Consequently, if $\tilde L\ge \frac{2M}{r_\Omega}$, then every step $\eta\in T_x\mathcal M$ satisfying $p_x(\eta)\le0$ necessarily belongs to $B(0_x,r_\Omega)$.
\end{lemma}
\begin{proof}
Since $\omega_{x_0}$ is compact and $\Omega$ is an open neighborhood containing $\omega_{x_0}$, a standard compactness argument yields the existence of a uniform radius $r_\Omega>0$ such that
\[
R_x(B(0_x,r_\Omega))
\subset
\Omega,
\qquad
\forall x\in\omega_{x_0}.
\]

By continuity of $\operatorname{grad} f_i$ on the compact set $\omega_{x_0}$, there exists $C_f>0$ such that
\[
\|\operatorname{grad} f_i(x)\|_x
\le
C_f,
\qquad
\forall x\in\omega_{x_0},\ i=1,\dots,m.
\]

Since each $g_i$ is locally Lipschitz on $\Omega$, shrinking $r_\Omega$ if necessary, there exists $C_g>0$ such that
\[
|g_i(R_x(\eta))-g_i(x)|
\le
C_g\|\eta\|_x
\]
for all $x\in\omega_{x_0}$ and $\eta\in B(0_x,r_\Omega)$.
Therefore,
\[
\langle \operatorname{grad} f_i(x),\eta\rangle_x
+
g_i(R_x(\eta))
-
g_i(x)
\ge
-(C_f+C_g)\|\eta\|_x.
\]

Setting $M:=C_f+C_g$ yields \eqref{eq:localization_bound}.
Then, 
\[
p_x(\eta)
\ge
-M\|\eta\|_x
+
\frac{\tilde L}{2}\|\eta\|_x^2.
\]
Hence, if $p_x(\eta)\le0$, then
\[
\frac{\tilde L}{2}\|\eta\|_x^2
\le
M\|\eta\|_x.
\]
For $\eta\neq0$, this implies
\[
\|\eta\|_x
\le
\frac{2M}{\tilde L}.
\]
Therefore, if $\tilde L\ge \frac{2M}{r_\Omega}$, then $\|\eta\|_x\le r_\Omega$.
\end{proof}
\begin{theorem}
    Suppose that Assumptions \ref{assume0}, \ref{assume1}, \ref{assume2}, and \ref{assume3} hold. Let $\{\varepsilon_k\}$ be a positive sequence such that $\lim_{k \to \infty} \varepsilon_k = 0$. Then the sequence $\{x_k\}$ has at least one Pareto accumulation point.
    Let $x_*$ be any accumulation point of the sequence $\{x_k\}$. Then $x_*$ is a Pareto stationary point.
\end{theorem}

\begin{proof}
Let $\eta_k^*$ be the exact minimizer of $p_{x_k}(\eta)$. By the algorithm's acceptance condition, we have $p_{x_k}(\hat{\eta}_k) \le p_{x_k}(0) = 0$, it follows that $\psi_{x_k}(\hat{\eta}_k) \le -\frac{\tilde{L}}{2}\|\hat{\eta}_k\|_{x_k}^2$. By the definition of $\psi_{x_k}$ as a pointwise maximum~\eqref{psi}, this implies that  every component satisfies the bound that
    \begin{equation}\label{eq:inexact-per-component}
        \langle \operatorname{grad} f_i(x_k), \hat{\eta}_k \rangle_{x_k} + g_i(R_{x_k}(\hat{\eta}_k)) - g_i(x_k) \le \psi_{x_k}(\hat{\eta}_k) \le -\frac{\tilde{L}}{2}\|\hat{\eta}_k\|_{x_k}^2, \quad i = 1,\ldots,m.
    \end{equation}
    Combining~\eqref{eq:inexact-per-component} with the $L$-retraction-smoothness of each $f_i$ (Assumption~\ref{assume2}) yields, for every $i = 1,\ldots,m$:
    \begin{equation}
    \begin{aligned}
        F_i(x_{k+1}) &\le F_i(x_k) - \frac{\tilde{L}}{2}\|\hat{\eta}_k\|_{x_k}^2 + \frac{L}{2}\|\hat{\eta}_k\|_{x_k}^2 = F_i(x_k) - \beta\|\hat{\eta}_k\|_{x_k}^2,
    \end{aligned}
    \end{equation}
    where $\beta = (\tilde{L} - L)/2 > 0$. This establishes $F_i(x_k) - F_i(x_{k+1}) \ge \beta\|\hat{\eta}_k\|_{x_k}^2$ for all $i$.

    Fixing any component $i\in\{1,\dots,m\}$, summing the descent inequality from $k=0$ to $N$ and using Assumption~\ref{assume1} gives
    \[
        \sum_{k=0}^{\infty} \|\hat{\eta}_k\|_{x_k}^2 < \infty, \qquad\text{hence}\qquad \lim_{k \to \infty} \|\hat{\eta}_k\|_{x_k} = 0.
    \]

    % To establish the relationship between the inexact step $\hat{\eta}_k$ and the exact minimizer $\eta_k^*$, we must first ensure they are confined to the retraction-convex neighborhood so that the strong convexity of $p_{x_k}$ is rigorously justified. 
    
    Provided that the penalty parameter is chosen such that $\tilde{L} \geq 2M/r_\Omega$, Lemma \ref{lem:localization} guarantees that both $\hat{\eta}_k$ and $\eta_k^*$ are strictly contained within the ball $B(0_{x_k}, r_\Omega)$. Within this localized ball, the mapped points $R_{x_k}(\eta)$ remain in $\Omega$, ensuring that the composition $g_i \circ R_{x_k}$ is a convex function. Since $\psi_{x_k}(\eta)$ is the pointwise maximum of convex functions on this ball, it is also convex. Consequently, the subproblem objective $p_{x_k}(\eta) = \psi_{x_k}(\eta) + \frac{\tilde{L}}{2}\|\eta\|_{x_k}^2$ is well-defined and strictly $\tilde{L}$-strongly convex on the domain $B(0_{x_k}, r_\Omega)$. 
    
    Evaluating the subgradients $v_k \in \partial p_{x_k}(\hat{\eta}_k)$ and $0 \in \partial p_{x_k}(\eta_k^*)$, and applying the Cauchy-Schwarz inequality, we obtain
    \begin{equation}
        \tilde{L}\|\hat{\eta}_k - \eta_k^*\|_{x_k}^2 \leq \langle v_k - 0, \hat{\eta}_k - \eta_k^* \rangle_{x_k} \leq \|v_k\|_{x_k} \|\hat{\eta}_k - \eta_k^*\|_{x_k}.
    \end{equation}
    Dividing by $\tilde{L}\|\hat{\eta}_k - \eta_k^*\|_{x_k}$ yields
    \begin{equation}\label{eq:sc-bound}
        \|\hat{\eta}_k - \eta_k^*\|_{x_k} \leq \frac{1}{\tilde{L}}\|v_k\|_{x_k}.
    \end{equation}

    Using the inexactness criterion $\|v_k\|_{x_k} \leq \varepsilon_k \|\hat{\eta}_k\|_{x_k}$,
    \[
        \|\eta_k^*\|_{x_k} \leq \|\hat{\eta}_k\|_{x_k} + \|\hat{\eta}_k - \eta_k^*\|_{x_k} \leq \|\hat{\eta}_k\|_{x_k} \Bigl( 1 + \frac{\varepsilon_k}{\tilde{L}} \Bigr) \to 0 \quad\text{as } k\to\infty.
    \]

    Let $\{x_{k_j}\}$ be a subsequence converging to $x_*$. Define the auxiliary points $y_{k_j} := R_{x_{k_j}}(\eta_{k_j}^*)$. Since $x_{k_j}\to x_*$ and $\|\eta_{k_j}^*\|_{x_{k_j}}\to0$, we have $y_{k_j}\to x_*$. For large $j$, $y_{k_j}$ lies in a compact neighborhood of $x_*\subseteq\omega_{x_0}$, so $\partial F_i(y_{k_j})$ is uniformly bounded  by Assumption~\ref{assume0} and~\cite[Proposition~2.5]{hosseini_line_2018}.

    At the exact minimizer $\eta_{k_j}^*$, the optimality condition $0\in\partial p_{x_{k_j}}(\eta_{k_j}^*)$ and the Riemannian chain rule for convex functions~\cite{hosseini2011generalized} give multipliers $\lambda^{(k_j)}\in\Delta_m$ and subgradients $\zeta_{i,y_{k_j}}\in\partial g_i(y_{k_j})$ such that
    \begin{equation}\label{eq:kkt-inexact}
        \sum_{i=1}^m \lambda_i^{(k_j)}\operatorname{grad} f_i(x_{k_j}) + \tilde L\eta_{k_j}^* + \sum_{i=1}^m \lambda_i^{(k_j)} (\operatorname{D} R_{x_{k_j}}(\eta_{k_j}^*))^*[\zeta_{i,y_{k_j}}] = 0.
    \end{equation}
    The chain rule $\partial(g_i\circ R_{x_{k_j}})(\eta) \subseteq (\operatorname{D} R_{x_{k_j}}(\eta))^*[\partial g_i(R_{x_{k_j}}(\eta))]$ is valid because $g_i$ is locally Lipschitz and $R_{x_{k_j}}$ is smooth.

    Define that
    \[
        v_{k_j} := \sum_{i=1}^m \lambda_i^{(k_j)}\bigl(\operatorname{grad} f_i(y_{k_j}) + \zeta_{i,y_{k_j}}\bigr) \in T_{y_{k_j}}\mathcal{M}.
    \]
    Let $\mathcal{T}_{k_j}:=\operatorname{D} R_{x_{k_j}}(\eta_{k_j}^*)\colon T_{x_{k_j}}\mathcal{M}\to T_{y_{k_j}}\mathcal{M}$ be the vector transport by differentiated retraction. Applying $\mathcal{T}_{k_j}$ to both sides of~\eqref{eq:kkt-inexact} and subtracting the result from $v_{k_j}$ yields the exact decomposition
    \begin{equation}\label{eq:vk-decomposition-inexact}
    \begin{aligned}
        v_{k_j} &= \sum_{i=1}^m \lambda_i^{(k_j)}\Bigl(\operatorname{grad} f_i(y_{k_j}) - \mathcal{T}_{k_j}\operatorname{grad} f_i(x_{k_j})\Bigr) \\
        &\quad + \sum_{i=1}^m \lambda_i^{(k_j)}\Bigl(\zeta_{i,y_{k_j}} - \mathcal{T}_{k_j}(\operatorname{D} R_{x_{k_j}}(\eta_{k_j}^*))^*[\zeta_{i,y_{k_j}}]\Bigr)
        \;-\; \tilde L\,\mathcal{T}_{k_j}\eta_{k_j}^*.
    \end{aligned}
    \end{equation}
    Now we estimate each term as $j\to\infty$:
    \begin{itemize}
        \item \textbf{First term.} Since $\operatorname{grad} f_i$ is a continuous vector field (by Assumption~\ref{assume2} on the compact set $\omega_{x_0}$) and $\mathcal{T}_{k_j}$ depends smoothly on its argument with $\mathcal{T}_{k_j}\to\operatorname{id}$ as $\|\eta_{k_j}^*\|\to0$, Lemma~\ref{lemma2} gives $\|\operatorname{grad} f_i(y_{k_j})-\mathcal{T}_{k_j}\operatorname{grad} f_i(x_{k_j})\|_{y_{k_j}}\to0$ for each $i$. The convex combination with weights $\lambda_i^{(k_j)}\in[0,1]$ also tends to $0$.
        \item \textbf{Second term.} Because $R_{x_{k_j}}$ is smooth and $\operatorname{D} R_{x_{k_j}}(0_{x_{k_j}})=\operatorname{id}$, we have $(\operatorname{D} R_{x_{k_j}}(\eta_{k_j}^*))^* = \operatorname{id} + \mathcal{O}(\|\eta_{k_j}^*\|)$ and $\mathcal{T}_{k_j}=\operatorname{id}+\mathcal{O}(\|\eta_{k_j}^*\|)$. Consequently $\mathcal{T}_{k_j}(\operatorname{D} R_{x_{k_j}}(\eta_{k_j}^*))^*=\operatorname{id}+\mathcal{O}(\|\eta_{k_j}^*\|)$. Hence
        \[
            \bigl\|\zeta_{i,y_{k_j}}-\mathcal{T}_{k_j}(\operatorname{D} R_{x_{k_j}}(\eta_{k_j}^*))^*[\zeta_{i,y_{k_j}}]\bigr\|_{y_{k_j}}
            \le C\|\eta_{k_j}^*\|_{x_{k_j}}\|\zeta_{i,y_{k_j}}\|_{y_{k_j}}
        \]
        for some constant $C>0$ depending only on the retraction. Since $\|\eta_{k_j}^*\|\to0$ and the subgradients $\zeta_{i,y_{k_j}}$ are uniformly bounded on the compact neighborhood of $x_*$, Term~B converges to $0$.
        \item \textbf{Third term.} $\|\mathcal{T}_{k_j}\eta_{k_j}^*\|_{y_{k_j}}\le\|\mathcal{T}_{k_j}\|\,\|\eta_{k_j}^*\|_{x_{k_j}}\to0$ as $\|\eta_{k_j}^*\|\to0$.
    \end{itemize}
    Therefore, we obtain
    \begin{equation}\label{eq:v-inexact-zero}
        \lim_{j\to\infty}\|v_{k_j}\|_{y_{k_j}} = 0.
    \end{equation}

    For contradiction, suppose that $x_*$ is not Pareto stationary. Then there exist a direction $d\in T_{x_*}\mathcal{M}$ and $\delta>0$ such that $\max_{i=1,\dots,m}F_i^\circ(x_*;d)<-2\delta$.
    Transport $d$ to $y_{k_j}$ via $d_{k_j}:=(\operatorname{D} R_{x_{k_j}}(\eta_{k_j}^*))^{*,-1}\bigl[\mathcal{T}_{x_*,x_{k_j}}(d)\bigr]\in T_{y_{k_j}}\mathcal{M}$ as in Theorem~\ref{thm:convergence}, so that $d_{k_j}\to d$ in the tangent bundle $T\mathcal{M}$ and $\|d_{k_j}\|$ stays bounded. By the upper semicontinuity of Clarke directional derivatives on $T\mathcal{M}$~\cite[Theorem~2.2]{hosseini_line_2018}, for sufficiently large $j$,
    \[
        \max_{i=1,\dots,m} F_i^\circ(y_{k_j}; d_{k_j}) \le -\delta.
    \]
    Since $\operatorname{grad} f_i(y_{k_j})+\zeta_{i,y_{k_j}}\in\partial F_i(y_{k_j})$, each component of $v_{k_j}$ satisfies $\langle\operatorname{grad} f_i(y_{k_j})+\zeta_{i,y_{k_j}},d_{k_j}\rangle\le F_i^\circ(y_{k_j};d_{k_j})\le -\delta$. Weighting and summing with $\lambda_i^{(k_j)}\ge0$ gives
    \[
        \langle v_{k_j}, d_{k_j}\rangle_{y_{k_j}} \le \sum_{i=1}^m \lambda_i^{(k_j)}F_i^\circ(y_{k_j};d_{k_j}) \le -\delta.
    \]
    However,~\eqref{eq:v-inexact-zero} and the boundedness of $\|d_{k_j}\|$ imply $|\langle v_{k_j},d_{k_j}\rangle|\le\|v_{k_j}\|\|d_{k_j}\|\to0$, contradicting the constant negative bound $-\delta$. Therefore $\max_i F_i^\circ(x_*;d)\ge0$ for every $d\in T_{x_*}\mathcal{M}$, i.e., $x_*$ is a Pareto stationary point.
\end{proof}
\begin{remark}
Huang and Wei~\cite{huang_inexact_2023} studied an inexact Riemannian proximal gradient method for single-objective composite optimization, using function-value and subgradient-based inexactness criteria and establishing local convergence rates under the Riemannian Kurdyka--{\L}ojasiewicz property. 
In contrast, we adopt the simpler condition
$
\|v_k\|_{x_k}\le\varepsilon_k\|\hat{\eta}_k\|_{x_k},
$
which suffices to guarantee global Pareto stationarity in the multiobjective setting. Discussing more general inexactness criteria and KL-type local convergence analysis to the multiobjective case is left for future work.
\end{remark}

\section{Method with trust region scheme}
\label{sec: trust region}

In this section, we develop an adaptive multiobjective Riemannian proximal gradient method. Existing Riemannian multiobjective trust-region methods, such as Eslami et al.~\cite{eslami2023trust}, mainly focus on smooth optimization problems with explicit trust-region constraints $\|\eta\|_{x_k}\le\Delta_k$. In contrast, our method addresses nonsmooth composite objectives and adopts a penalty-based trust-region strategy as Zhao et al.~\cite{zhao_proximal_2023}, where a dynamic regularization parameter $\sigma_k$ plays a role analogous to the inverse trust-region radius.

At each iteration $k$, the search direction $\eta_k$ is generated by minimizing the strongly convex quadratic trust-region subproblem on the tangent space $T_{x_k}\mathcal{M}$:
\begin{equation}\label{subproblem2}
    m_{x_k}(\eta, \sigma_k) := \psi_{x_k}(\eta) + \frac{\sigma_k}{2} \big\|\eta \big\|_{x_k}^{2},
\end{equation}
where $\sigma_k > 0$ is the trust-region parameter, and $\psi_{x_k}$ is the multiobjective first-order approximation defined in \eqref{psi}. Since $m_{x_k}(\eta, \sigma_k)$ is a $\sigma_k$-strongly convex function, it admits a unique exact minimizer $\eta_k$.
Once a trial step $\eta_k$ is computed by minimizing $m_{x_k}(\eta, \sigma_k)$, the quality of the model is assessed by the ratio of the actual reduction to the predicted reduction. 
In the multiobjective trust-region scheme, we first solve the subproblem to obtain the search direction $\eta_k$. If $\eta_k = 0$, Lemma \ref{lemma: P-sta} guarantees that $x_k$ is a Pareto stationary point, and the algorithm terminates immediately. Otherwise, assuming $\eta_k \neq 0$, to ensure Pareto descent without degrading any individual objective, the ratio at the $k$-th iteration is evaluated based on the least-decreased component:
\begin{equation}\label{eq rhok}
\rho_k = \frac{\min_{i=1,\dots,m} \{F_i(x_k) - F_i(R_{x_k}(\eta_k))\}}{-m_{x_k}(\eta_k, \sigma_k)}.
\end{equation}
Since $\eta_k$ is the exact minimizer of the $\sigma_k$-strongly convex local model $m_{x_k}$, evaluating the strong convexity inequality at the origin $0 \in T_{x_k}\mathcal{M}$ yields 
\begin{equation}\label{eq:predicted-reduction} 
-m_{x_k}(\eta_k, \sigma_k) \ge \frac{\sigma_k}{2}\|\eta_k\|_{x_k}^2.
\end{equation}
Given that the algorithm only proceeds to evaluate $\rho_k$ when $\eta_k \neq 0$, we have strictly $\|\eta_k\|_{x_k} > 0$. This guarantees that the predicted reduction is well-defined and strictly positive.

    \begin{algorithm}[H]
      \caption{Trust-Region Riemannian Multiobjective Proximal Gradient Method (TR-RMPGM)}
      \label{alg:ts-pg}
        \noindent \textbf{Input:} Initial iterate $x_0 \in \mathcal{M}$; initial penalty parameter $\sigma_0 > \sigma_{\min} > 0$; parameters $0 < \tau_1 < 1 < \tau_2 < \tau_3$ and $0 < s_1 < s_2 < 1$. \\
        1: for $k=0, 1, \ldots$ do\\
        2: \quad Compute $\eta_k = \operatorname{argmin}_{\eta \in T_{x_k}\mathcal{M}} m_{x_k}(\eta, \sigma_k)$; \\
        3: \quad Calculate the ratio $\rho_k$ using~\eqref{eq rhok}; \\
        4: \quad if $\rho_k \ge s_1$ then $x_{k+1} = R_{x_k}(\eta_k)$; else $x_{k+1} = x_k$; \\
        5: \quad Update $\sigma_{k+1}$: \\
            \qquad if $\rho_k \ge s_2$ (very successful), $\sigma_{k+1} \in [\max(\sigma_{\min}, \tau_1 \sigma_k), \sigma_k]$;\\
            \quad else if $\rho_k < s_1$ (unsuccessful), $\sigma_{k+1} \in [\tau_2 \sigma_k, \tau_3 \sigma_k]$;\\
            \quad else (successful), $\sigma_{k+1} \in [\sigma_k, \tau_2 \sigma_k]$. \\
        6: end for
    \end{algorithm}

\subsection{Convergence analysis}
For the convergence analysis of the trust-region scheme, it is convenient to employ the absolute-value formulation of $L$-retraction-smoothness. Under the combination of Assumptions~\ref{assume2} and~\ref{assume3}, the bilateral bound
    \begin{equation}\label{eq: L-retrac-smo}
            \left|f_i(R_x(\eta)) - f_i(x) - \langle\operatorname{grad} f_i(x), \eta\rangle_x \right| \leq \frac{L}{2}\|\eta\|_x^2
    \end{equation}
holds for each $i$. Indeed, Assumption~\ref{assume2} provides the upper estimate $f_i(R_x(\eta)) - f_i(x) - \langle\operatorname{grad} f_i(x),\eta\rangle_x \le \frac{L}{2}\|\eta\|_x^2$, while retraction-convexity gives the lower bound $f_i(R_x(\eta)) - f_i(x) - \langle\operatorname{grad} f_i(x),\eta\rangle_x \ge 0$. Together they imply the two-sided inequality~\eqref{eq: L-retrac-smo}. 

The following lemma establishes the connection between the trust-region subproblem and Pareto stationarity. We omit the proof here, as it closely parallels the previous arguments.
\begin{lemma}\label{lemma: P-sta}
    Let $\{\eta_k\}$ and $\{x_k\}$ be generated by Algorithm \ref{alg:ts-pg}. If $\eta_k = 0$ for some $k$, then $x_k$ is a Pareto stationary point.
\end{lemma}

  We can now establish that whenever $\sigma_k > \sigma_{\text{succ}}$ for a fixed constant $\sigma_{\text{succ}}$, the iteration is necessarily successful.
    \begin{theorem}
    Let $\{x_k\}$ be the sequence generated by Algorithm \ref{alg:ts-pg}. Suppose Assumptions \ref{assume1}, \ref{assume2}, and \ref{assume3} hold, and let $\sigma_{succ} = L / (1 - s_2) > 0$. If $x_k$ is not a Pareto stationary point and $\sigma_k \geq \sigma_{succ}$, then the $k$-th iteration is very successful, i.e., $\rho_k \geq s_2$.
    \end{theorem}
    \begin{proof}
        % Given that $\eta_k$ is the exact minimizer of $m_{x_k}$, evaluating the strong convexity inequality at the origin $0 \in T_{x_k}\mathcal{M}$ directly yields:
        % \begin{equation}
        %     m_{x_k}(0, \sigma_k) \ge m_{x_k}(\eta_k, \sigma_k) + \frac{\sigma_k}{2}\|\eta_k - 0\|_{x_k}^2.
        % \end{equation}
        % Since $m_{x_k}(0, \sigma_k) = \psi_{x_k}(0) = 0$, this directly bounds the predicted reduction:
        % \begin{equation}\label{eq:pred_red}
        %     \red{-m_{x_k}(\eta_k, \sigma_k) \ge \frac{\sigma_k}{2}\|\eta_k\|_{x_k}^2.}
        % \end{equation}

        By Lemma \ref{lemma: P-sta}, $\eta_k \neq 0$ if $x_k$ is not stationary. 
        For each individual component $i$,  form~\eqref{eq: L-retrac-smo}, we have the lower bound
        \[
            f_i(x_k) - f_i(R_{x_k}(\eta_k)) \ge -\langle \operatorname{grad} f_i(x_k), \eta_k\rangle_{x_k} - \frac{L}{2}\|\eta_k\|_{x_k}^2.
        \]
        Adding $g_i(x_k) - g_i(R_{x_k}(\eta_k))$ to both sides, we obtain for \emph{every} $i = 1,\dots,m$:
        \begin{equation}\label{eq:component-wise-lower}
            F_i(x_k) - F_i(R_{x_k}(\eta_k)) \ge -\Bigl(\langle \operatorname{grad} f_i(x_k), \eta_k\rangle_{x_k} + g_i(R_{x_k}(\eta_k)) - g_i(x_k)\Bigr) - \frac{L}{2}\|\eta_k\|_{x_k}^2.
        \end{equation}
        From the definition of $\psi_{x_k}(\eta_k)$ in~\eqref{psi}, we have for every $i$: $\langle \operatorname{grad} f_i(x_k), \eta_k\rangle_{x_k} + g_i(R_{x_k}(\eta_k)) - g_i(x_k) \le \psi_{x_k}(\eta_k)$. Therefore the bracketed term in~\eqref{eq:component-wise-lower} is bounded above by $\psi_{x_k}(\eta_k)$, then
        \begin{equation}\label{eq:component-wise-uniform}
            F_i(x_k) - F_i(R_{x_k}(\eta_k)) \ge -\psi_{x_k}(\eta_k) - \frac{L}{2}\|\eta_k\|_{x_k}^2, \qquad \forall\, i = 1,\dots,m.
        \end{equation}
        Taking the minimum over all $i$ preserves the inequality:
        \begin{equation}\label{eq:actual-lower}
            \min_{i=1,\dots,m}\{F_i(x_k) - F_i(R_{x_k}(\eta_k))\} \ge -\psi_{x_k}(\eta_k) - \frac{L}{2}\|\eta_k\|_{x_k}^2.
        \end{equation}
        
        Now evaluate the ratio $\rho_k$ using the actual reduction bound~\eqref{eq:actual-lower} and the predicted reduction $-m_{x_k}(\eta_k, \sigma_k)$, recall~\eqref{eq:predicted-reduction} that $-m_{x_k}(\eta_k, \sigma_k) \ge \frac{\sigma_k}{2}\|\eta_k\|_{x_k}^2>0$, we have
        \begin{equation}\label{eq:ratio-one-sided}
        \begin{aligned}
            \rho_k \;=\; \frac{\min_i\{F_i(x_k) - F_i(R_{x_k}(\eta_k))\}}{-m_{x_k}(\eta_k, \sigma_k)}
            &\;\ge\; \frac{-\psi_{x_k}(\eta_k) - \frac{L}{2}\|\eta_k\|_{x_k}^2}{-m_{x_k}(\eta_k, \sigma_k)} \\
            &\;=\; \frac{-\psi_{x_k}(\eta_k)}{-m_{x_k}(\eta_k, \sigma_k)} - \frac{L}{2} \frac{\|\eta_k\|_{x_k}^2}{-m_{x_k}(\eta_k, \sigma_k)}.
        \end{aligned}
        \end{equation}
        Since $-m_{x_k}(\eta_k, \sigma_k) = -\psi_{x_k}(\eta_k) - \frac{\sigma_k}{2}\|\eta_k\|_{x_k}^2 \le -\psi_{x_k}(\eta_k)$, the first term $\frac{-\psi_{x_k}(\eta_k)}{-m_{x_k}(\eta_k, \sigma_k)} > 1$. For the second term, $x_k$ is not Pareto stationary, so Lemma~\ref{lemma: P-sta} implies $\eta_k\neq0$, rendering the denominator strictly positive. 
        % Substituting the strong convexity bound~\eqref{eq:predicted-reduction} (which asserts $-m_{x_k}(\eta_k, \sigma_k) \ge \frac{\sigma_k}{2}\|\eta_k\|_{x_k}^2$), we obtain
        Then,
        \begin{equation}
            \rho_k \ge 1 - \frac{L}{2} \frac{2}{\sigma_k} = 1 - \frac{L}{\sigma_k}.
        \end{equation}
        Thus, whenever $\sigma_k \ge \frac{L}{1-s_2}$, we are guaranteed that $\rho_k \ge 1 - (1-s_2) = s_2$, meaning the $k$-th iteration is very successful.
    \end{proof}

    This theorem ensures that Algorithm \ref{alg:ts-pg} prevents infinite cycling within its inner loop. We now show that the algorithm globally converges to a Pareto stationary point.

    \begin{theorem}\label{thm:tr-finite}
        Let $\{x_k\}$ be the sequence generated by Algorithm \ref{alg:ts-pg}. Suppose Assumptions \ref{assume0}, \ref{assume1}, \ref{assume2}, and \ref{assume3} are satisfied. If Algorithm \ref{alg:ts-pg} generates only finitely many successful iterations, then $x_k = x^*$ for all sufficiently large $k$, where $x^*$ is a Pareto stationary point.
    \end{theorem}
    \begin{proof}
        Suppose the last successful iteration occurs at index $K$. Then for all $k \ge K$, we have $x_k = x^* := x_K$. Since every iteration $k \ge K$ is unsuccessful, the update rule dictates $\sigma_{k+1} \in [\tau_2 \sigma_k, \tau_3 \sigma_k]$ with $\tau_2 > 1$, hence $\sigma_k \to \infty$.
        Suppose for contradiction that $x^*$ is not Pareto stationary. Then $\eta_k \neq 0$ for all $k \ge K$ by Lemma~\ref{lemma: P-sta}. Since $\sigma_k \to \infty$, eventually $\sigma_k \ge \sigma_{\text{succ}} = L/(1-s_2)$. By the preceding theorem, this iteration must be very successful, contradicting the assumption that all iterations $k \ge K$ are unsuccessful. Therefore, $x^*$ must be a Pareto stationary point.
    \end{proof}

    We now consider the case of infinitely many successful iterations. The following result ensures the global convergence of the algorithm to a Pareto stationary point.

    \begin{proposition}[Uniform boundedness of $\sigma_k$]\label{prop:sigma-bound}
        Under the conditions of Algorithm~\ref{alg:ts-pg}, the penalty parameter satisfies $\sigma_k \le \sigma_{\max}$ for all $k$, where $\sigma_{\max} := \max\{\sigma_0,\, \tau_3 \sigma_{\text{succ}}\}$ and $\sigma_{\text{succ}} = L/(1-s_2)$.
    \end{proposition}
    \begin{proof}
        We show by induction that $\sigma_k \le \sigma_{\max}$ for all $k$. The base case $\sigma_0 \le \sigma_{\max}$ holds by definition. For the inductive step, suppose $\sigma_k \le \sigma_{\max}$. There are three cases:
        \begin{itemize}
            \item \textbf{Very successful} ($\rho_k \ge s_2$): $\sigma_{k+1} \le \sigma_k \le \sigma_{\max}$.
            \item \textbf{Successful} ($s_1 \le \rho_k < s_2$): $\sigma_{k+1} \le \tau_2 \sigma_k$. If $\sigma_k < \sigma_{\text{succ}}$, this is allowed. If $\sigma_k \ge \sigma_{\text{succ}}$, the preceding theorem guarantees $\rho_k \ge s_2$, contradicting $\rho_k < s_2$. So this case only occurs when $\sigma_k < \sigma_{\text{succ}}$, and $\sigma_{k+1} \le \tau_2 \sigma_{\text{succ}} < \tau_3 \sigma_{\text{succ}} \le \sigma_{\max}$.
            \item \textbf{Unsuccessful} ($\rho_k < s_1$): $\sigma_{k+1} \le \tau_3 \sigma_k$. Similarly, $\sigma_k < \sigma_{\text{succ}}$, otherwise the iteration would be very successful, so $\sigma_{k+1} \le \tau_3 \sigma_{\text{succ}} \le \sigma_{\max}$.
        \end{itemize}
        In all cases, $\sigma_{k+1} \le \sigma_{\max}$, completing the induction.
    \end{proof}
\begin{lemma}\label{lemma:monotonicity}
    Let $\psi_x \colon T_x\mathcal{M} \to \mathbb{R} \cup \{+\infty\}$ be a proper, lower semicontinuous, and convex function. Let $\eta(\sigma)$ be the unique minimizer of $m_x(\eta, \sigma) := \psi_x(\eta) + \frac{\sigma}{2}\|\eta\|_x^2$ for $\sigma > 0$. Then for any $0 < \sigma_1 \le \sigma_2$, we have
    \begin{equation}
        \|\eta(\sigma_1)\|_x \le \frac{\sigma_2}{\sigma_1} \|\eta(\sigma_2)\|_x.
    \end{equation}
\end{lemma}
\begin{proof}
    For $j\in\{1,2\}$, let $\eta_j=\eta(\sigma_j)$. The optimality condition gives
    $
        0\in \partial\psi_x(\eta_j)+\sigma_j\eta_j,
    $
    so there exists $u_j\in\partial\psi_x(\eta_j)$ such that
    $
        u_j=-\sigma_j\eta_j.
    $
    By the monotonicity of the convex subdifferential,
    \[
        \langle u_1-u_2,\eta_1-\eta_2\rangle_x\ge0.
    \]
    Substituting the expressions for $u_1$ and $u_2$ yields
    \[
        \langle -\sigma_1\eta_1+\sigma_2\eta_2,\eta_1-\eta_2\rangle_x\ge0,
    \]
    and therefore
    \[
        \sigma_1\|\eta_1\|_x^2+\sigma_2\|\eta_2\|_x^2
        \le
        (\sigma_1+\sigma_2)\|\eta_1\|_x\|\eta_2\|_x.
    \]
    Equivalently,
    \[
        (\|\eta_1\|_x-\|\eta_2\|_x)
        (\sigma_1\|\eta_1\|_x-\sigma_2\|\eta_2\|_x)
        \le0.
    \]
    Since $\sigma_1\le\sigma_2$, it follows that
    \[
        \|\eta_1\|_x
        \le
        \frac{\sigma_2}{\sigma_1}\|\eta_2\|_x.
    \]
\end{proof}
    \begin{theorem}\label{thm:tr-infinite}
        Let $\{x_k\}$ be the sequence generated by Algorithm \ref{alg:ts-pg}. Suppose Assumptions \ref{assume0}, \ref{assume1}, \ref{assume2}, and \ref{assume3} are satisfied. If Algorithm \ref{alg:ts-pg} generates infinitely many successful iterations, then
        $$
        \lim _{k \rightarrow \infty}\left\|\eta_k\right\|_{x_k}=0,
        $$
        and any accumulation point $x_*$ of $\{x_k\}$ is a Pareto stationary point.
    \end{theorem} 
    \begin{proof}

        Let $k \in \mathbb{S}$ be a successful iteration, i.e., $\rho_k \ge s_1$. By the definition of $\rho_k$ and the strong convexity of the local model, we have
$$
\min_{i=1,\ldots,m} \{F_i(x_k) - F_i(x_{k+1})\} \ge s_1 ( -m_{x_k}(\eta_k, \sigma_k) ) \ge \frac{s_1 \sigma_k}{2} \|\eta_k\|_{x_k}^2 \ge \frac{s_1 \sigma_{\min}}{2} \|\eta_k\|_{x_k}^2.
$$

Hence, for every $i=1,\dots,m$ and every successful iteration $k\in\mathbb S$,
\[
F_i(x_k)-F_i(x_{k+1})
\ge
\frac{s_1\sigma_{\min}}{2}\|\eta_k\|_{x_k}^2.
\]
Since each objective function $F_i$ is bounded from below by $F_i^{low}$ (Assumption \ref{assume1}), summing the above inequality over all iterations from $k=0$ to $K$ and letting $K \to \infty$ yields
$$
F_i(x_0) - F_i^{low} \ge \sum_{k=0}^\infty (F_i(x_k) - F_i(x_{k+1})) \ge \frac{s_1 \sigma_{\min}}{2} \sum_{k \in \mathbb{S}} \|\eta_k\|_{x_k}^2.
$$
Therefore, it immediately follows that $\sum_{k \in \mathbb{S}} \|\eta_k\|_{x_k}^2 < \infty$. Consequently, we obtain $\lim_{k \to \infty, k \in \mathbb{S}} \|\eta_k\|_{x_k} = 0$.

Next, we control the unsuccessful iterations $k \notin \mathbb{S}$. Because $x_{k+1} = x_k$, the base point remains unchanged. Let $\ell(k) \in \mathbb{S}$ be the first successful iteration after $k$. Since the number of consecutive unsuccessful iterations is strictly bounded by $\lceil \log_{\tau_2}(\sigma_{\max}/\sigma_{\min})\rceil$, we have $\sigma_k \le \sigma_{\ell(k)} \le \sigma_{\max}$ and $x_k = x_{\ell(k)}$. 
Applying Lemma~\ref{lemma:monotonicity} gives
\[
\|\eta_k\|_{x_k}
\le
\frac{\sigma_{\ell(k)}}{\sigma_k}
\|\eta_{\ell(k)}\|_{x_{\ell(k)}}
\le
\frac{\sigma_{\max}}{\sigma_{\min}}
\|\eta_{\ell(k)}\|_{x_{\ell(k)}}.
\]
Since $\ell(k)\to\infty$ as $k\to\infty$ and
$\|\eta_\ell\|_{x_\ell}\to0$ along successful iterations,
we conclude that
\[
\lim_{k\to\infty}\|\eta_k\|_{x_k}=0.
\]

        Let $x_*$ be an accumulation point of $\{x_k\}$ and let $\{x_{k_j}\}$ be a subsequence converging to $x_*$. Since $\|\eta_{k_j}\|_{x_{k_j}}\to0$, the remainder of the proof follows the argument of Theorem~\ref{thm:convergence}, with $\tilde L$ replaced by $\sigma_{k_j}$. Since $\{\sigma_k\}$ is uniformly bounded above, the additional term $\sigma_{k_j}\eta_{k_j}$ vanishes as $j\to\infty$. Therefore, $x_*$ is a Pareto stationary point.
    \end{proof}
 
\subsection{Iteration complexity analysis}
In this section, we establish the iteration complexity of Algorithm \ref{alg:ts-pg}. Let $\epsilon \in (0, 1)$ be a given tolerance. We define $\|\eta_k\|_{x_k}$ as our measure of Pareto stationarity, and the algorithm terminates when $\|\eta_k\|_{x_k} \le \epsilon$. 
Let $\mathbb{S} := \{ k \in \mathbb{N} \mid \rho_k \ge s_1 \}$ denote the set of successful iterations. We seek to bound the number of successful iterations $\mathbb{S}(\epsilon)$ and unsuccessful iterations $\mathbb{F}(\epsilon)$ performed before the stopping criterion is met.

\begin{lemma}[Complexity of successful iterations]
    Suppose Assumptions \ref{assume1}, \ref{assume2}, and \ref{assume3} hold. If $\{F_i(x_k)\}$ is bounded from below by $F_i^{low}$ for all $i = 1, \dots, m$, then the number of successful iterations is bounded by:
    \begin{equation}
        |\mathbb{S}(\epsilon)| \le \frac{2}{\epsilon^2 s_1 \sigma_{\min}} \min_{i=1,\dots,m} \left( F_i(x_0) - F_i^{low} \right) = \mathcal{O}(\epsilon^{-2}).
    \end{equation}
\end{lemma}
\begin{proof}
For any successful iteration $k\in\mathbb S(\epsilon)$,
\[
\min_i\{F_i(x_k)-F_i(x_{k+1})\}
\ge
s_1(-m_{x_k}(\eta_k,\sigma_k))
\ge
\frac{s_1\sigma_{\min}}{2}\|\eta_k\|_{x_k}^2
\ge
\frac{s_1\sigma_{\min}}{2}\epsilon^2.
\]
Hence, for every $i=1,\dots,m$,
\[
F_i(x_k)-F_i(x_{k+1})
\ge
\frac{s_1\sigma_{\min}}{2}\epsilon^2.
\]
Summing over all successful iterations gives
\[
F_i(x_0) - F_i^{low} \ge \sum_{k \in \mathbb{S}(\epsilon)} \left( F_i(x_k) - F_i(x_{k+1}) \right) \ge |\mathbb{S}(\epsilon)| \frac{s_1 \sigma_{\min}}{2} \epsilon^2.
\]
Taking the minimum over $i$ yields the result.
\end{proof}

To bound the unsuccessful iterations, we invoke Proposition~\ref{prop:sigma-bound}, which establishes that $\sigma_k \le \sigma_{\max} := \max\{\sigma_0, \tau_3 \sigma_{\text{succ}}\}$ for all $k$.

\begin{lemma}[Complexity of unsuccessful iterations]
    Under the same assumptions, the number of unsuccessful iterations before reaching $\epsilon$-stationarity satisfies:
    \begin{equation}
        |\mathbb{F}(\epsilon)| \le \log_{\tau_2} \left( \frac{\sigma_{\max}}{\sigma_0} \right) + |\mathbb{S}(\epsilon)| | \log_{\tau_2}(\tau_1) | = \mathcal{O}(\epsilon^{-2}).
    \end{equation}
\end{lemma}
\begin{proof}
By the update rules, we have
\[
\sigma_{k+1}\ge\tau_1\sigma_k,
\qquad k\in\mathbb S(\epsilon),
\]
and
\[
\sigma_{k+1}\ge\tau_2\sigma_k,
\qquad k\in\mathbb F(\epsilon).
\]
Hence,
\[
\sigma_{k(\epsilon)}
\ge
\sigma_0
\tau_1^{|\mathbb S(\epsilon)|}
\tau_2^{|\mathbb F(\epsilon)|}.
\]
Since $\sigma_k\le\sigma_{\max}$ for all $k$,
\[
\sigma_{\max}
\ge
\sigma_0
\tau_1^{|\mathbb S(\epsilon)|}
\tau_2^{|\mathbb F(\epsilon)|}.
\]
Taking logarithms yields
\[
|\mathbb F(\epsilon)|
\le
\log_{\tau_2}\left(\frac{\sigma_{\max}}{\sigma_0}\right)
+
|\mathbb S(\epsilon)|
|\log_{\tau_2}(\tau_1)|.
\]
The conclusion follows from the previous lemma.
\end{proof}

Combining the two lemmas immediately yields the global iteration complexity for the multiobjective trust-region scheme.
\begin{theorem}
    Suppose Assumptions \ref{assume1}, \ref{assume2}, and \ref{assume3} hold. Algorithm \ref{alg:ts-pg} requires at most $\mathcal{O}(\epsilon^{-2})$ total iterations to return an iterate satisfying the Pareto stationarity measure $\|\eta_k\|_{x_k} \le \epsilon$.
\end{theorem}
\begin{remark}
The trust-region ratio in \eqref{eq rhok} employs the full regularized model reduction
$
-m_{x_k}(\eta_k,\sigma_k)
=
-\psi_{x_k}(\eta_k)
-\frac{\sigma_k}{2}\|\eta_k\|_{x_k}^2
$
in the denominator, following the classical trust-region framework.
Alternatively, inspired by Zhao et al.~\cite{zhao_proximal_2023}, one may define the ratio using only the first-order model reduction:
\[
\widehat{\rho}_k
=
\frac{
\min_i\{F_i(x_k)-F_i(R_{x_k}(\eta_k))\}
}{
-\psi_{x_k}(\eta_k)
}.
\]
Using
$
-\psi_{x_k}(\eta_k)
\ge
\frac{\sigma_k}{2}\|\eta_k\|_{x_k}^2,
$
one obtains the same lower bound
$
\widehat{\rho}_k
\ge
1-\frac{L}{\sigma_k},
$
and therefore the same global convergence and iteration complexity results.
However, the present formulation based on the full model reduction aligns more naturally with the classical trust-region philosophy and leads to a slightly simpler analysis in the multiobjective setting.
\end{remark}

\section{Numerical experiments}
\label{sec: expreriment}

All experiments were implemented in Python~3.12. on a Macmini equipped with an Apple M4 chip and 16\,GB of RAM, running macOS~15. 
We evaluate the three RMPGM variants together with the Riemannian Multiobjective Steepest Descent (RMSD) baseline~\cite{fliege2000steepest} with subgradient.

% \subsection{Experiment 1: Bi-Objective Sparse Signal Recovery on the Sphere}

In applications such as robust sensor localization and multimodal signal recovery, one often seeks a sparse signal that simultaneously fits multiple measurement models. Let $A_1,A_2 \in \mathbb{R}^{m\times n}$ be two measurement matrices with corresponding noisy observations $b_1,b_2 \in \mathbb{R}^m$.

We formulate the bi-objective sparse signal recovery problem as follows
\begin{equation}
\label{eq:biobjective_sphere}
\min _{x \in \mathbb{S}^{n-1}} \left( \frac{1}{2}\left\|A_1 x - b_1\right\|_2^2 + \lambda_1\|x\|_1, \quad \frac{1}{2}\left\|A_2 x - b_2\right\|_2^2 + \lambda_2\|x\|_1 \right),
\end{equation}
where $\mathbb{S}^{n-1}=\left\{x \in \mathbb{R}^{n} \mid \|x\|_2=1\right\}$ denotes the unit sphere manifold. Here, $f_1(x) = \frac{1}{2}\|A_1 x - b_1\|_2^2$ and $f_2(x) = \frac{1}{2}\|A_2 x - b_2\|_2^2$ are smooth data fidelity functions, and the nonsmooth terms $g_1(x) = \lambda_1\|x\|_1$ and $g_2(x) = \lambda_2\|x\|_1$ enforce sparsity on the recovered signal.

\textbf{Experimental Setup.} 
We considered three problem sizes $(n,m)\in\{(128,50),(256,100),(512,200)\}$ with sparsity level $5\%$. To induce a clear conflict between objectives, the ground-truth signals $x_1^*$ and $x_2^*$ were generated with disjoint supports. The matrices $A_1,A_2$ were sampled from $\mathcal{N}(0,1/\sqrt{m})$, and observations were generated by $b_i=A_i x_i^*+\epsilon_i$ with small Gaussian noise. We set $\lambda_1=\lambda_2=0.05$.
All methods used $\mathrm{max\_iter}=500$ and tolerance $10^{-4}$. RMPGM and Inexact-RMPGM used $\tilde{L}=L$, while TR-RMPGM used $\sigma_0=L$. The inexact variant adopted $\varepsilon_k=1/(k+1)$. For TR-RMPGM, we set $s_1=0.1$, $s_2=0.75$, $\tau_1=0.5$, $\tau_2=2$, $\tau_3=4$, and $\sigma_{\min}=10^{-6}$. RMSD used the diminishing step size $\alpha_k=1/\sqrt{k+1}$.

% \textbf{Pareto Front Tracing Strategy.} 
% The Pareto front was traced by varying the initialization
% \[
% x_0=\mathrm{Proj}_{\mathbb{S}^{n-1}}
% \bigl(\gamma x_1^*+(1-\gamma)x_2^*+\xi\bigr),
% \]
% where $\gamma\in[0,1]$ controls the initialization bias and $\xi\sim\mathcal{N}(0,0.05^2I)$ is a small perturbation. Ten uniformly spaced values of $\gamma$ were tested. We compared the three RMPGM variants with the Riemannian Multiobjective Steepest Descent (RMSD) method~\cite{fliege2000steepest} using subgradients.

\begin{figure}[htpb]
    \centering
    \includegraphics[width=0.5\textwidth]{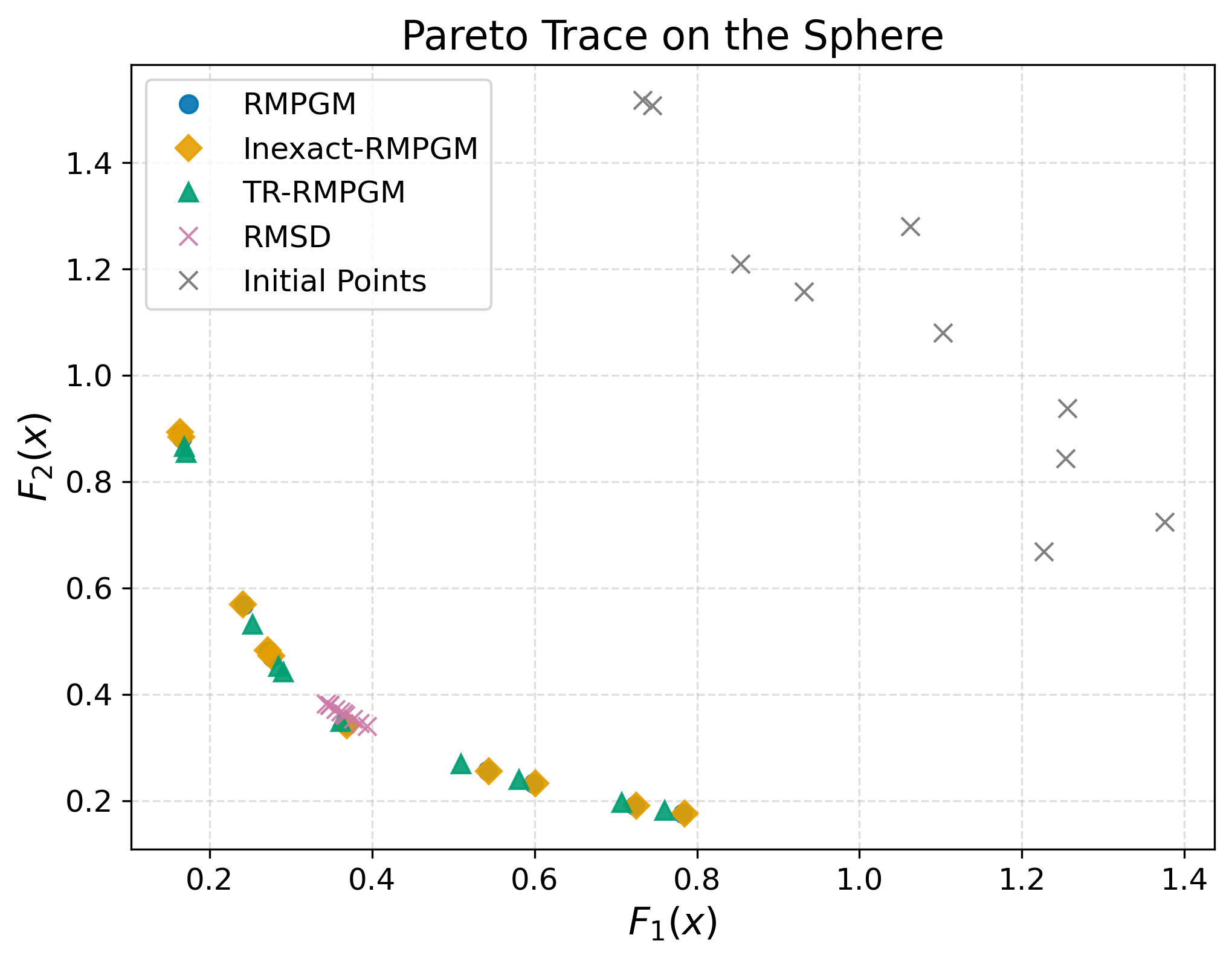}
    \caption{Pareto front obtained for the Bi-Objective Sparse Signal Recovery problem on the sphere with $n=256$ and $m=100$. }
    \label{fig:pareto_front}
\end{figure}

% \begin{figure}[htpb]
%     \centering
%     \includegraphics[width=0.5\textwidth]{Sphere_convergence.png}
%     \caption{Convergence behavior of the optimality measure for Experiment~1. }
%     \label{fig:convergence}
% \end{figure}

% \begin{figure}[htpb]
%     \centering
%     \includegraphics[width=0.5\textwidth]{Sphere_time_comparison.png}
%     \caption{Optimality measure versus wall-clock time for Experiment~1. }
%     \label{fig:time_comparison}
% \end{figure}

\begin{figure}[htpb]
    \centering
    \begin{minipage}{0.48\textwidth}
        \centering
        \includegraphics[width=\textwidth]{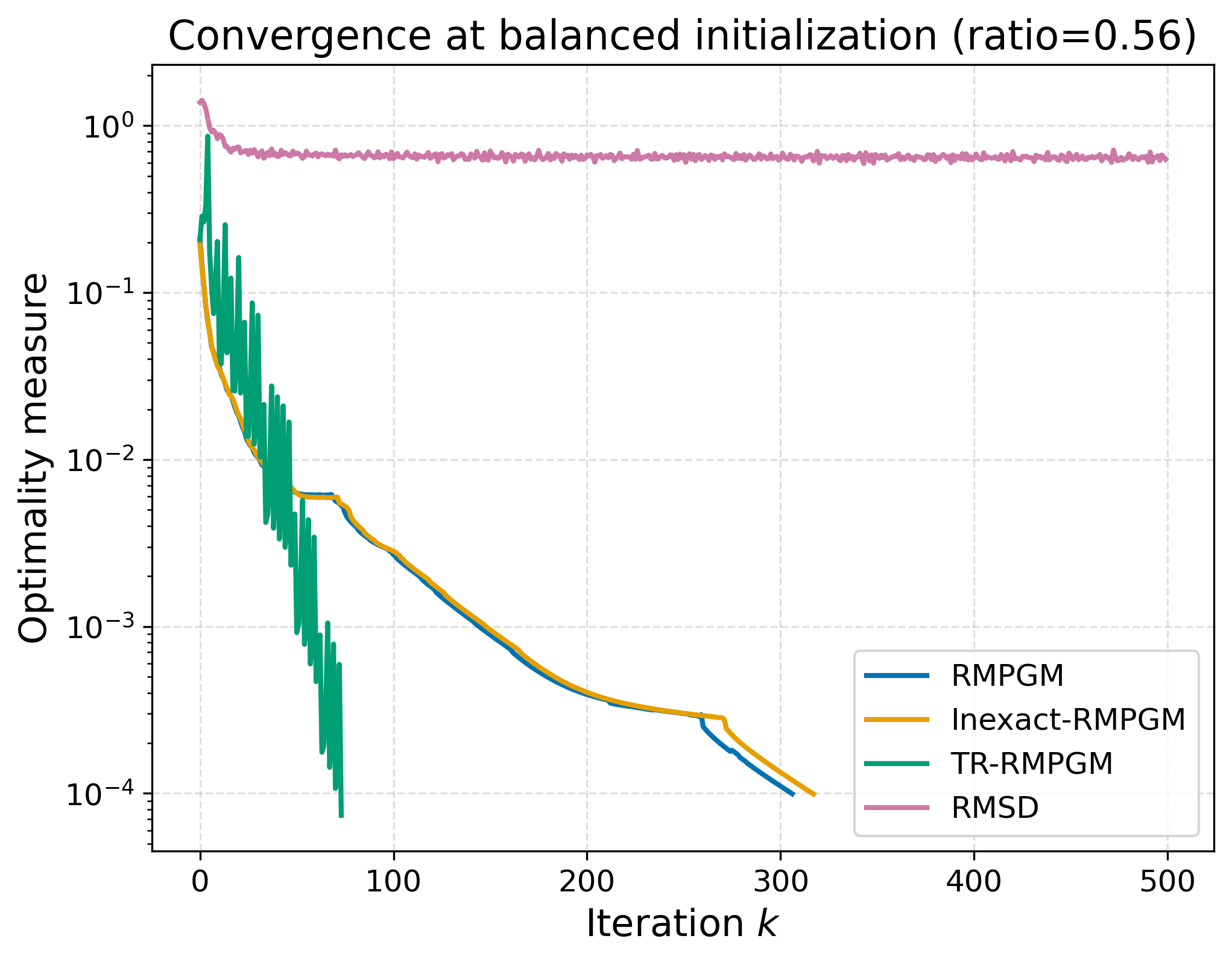}
        \caption{Convergence behavior of the optimality measure for Experiment~1 with $n=256$ and $m=100$.}
        \label{fig:convergence}
    \end{minipage}
    \hfill
    \begin{minipage}{0.48\textwidth}
        \centering
        \includegraphics[width=\textwidth]{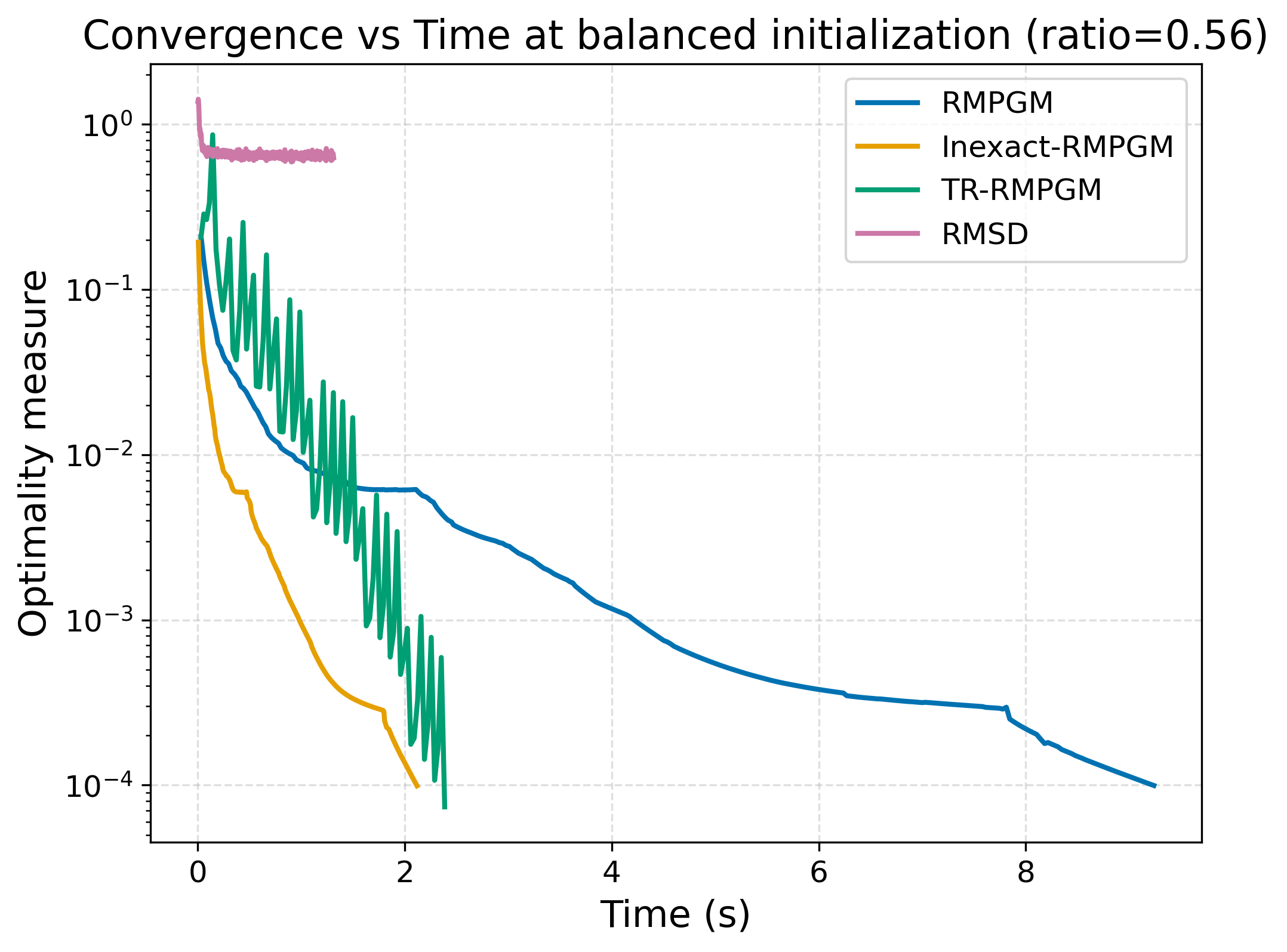}
        \caption{Optimality measure versus wall-clock time for Experiment~1 with $n=256$ and $m=100$.}
        \label{fig:time_comparison}
    \end{minipage}
\end{figure}

\begin{table}[htbp]
\centering
\small
\caption{Summary of Experiment~1 (Sphere) across different problem sizes.}
\label{tab:summary_sphere}

\begin{tabular}{lcccccc}
\hline
\textbf{Method}
& \multicolumn{2}{c}{$n=128,\ m=50$}
& \multicolumn{2}{c}{$n=256,\ m=100$}
& \multicolumn{2}{c}{$n=512,\ m=200$} \\

% \hline

& \textbf{Iters} & \textbf{Time(s)}
& \textbf{Iters} & \textbf{Time(s)}
& \textbf{Iters} & \textbf{Time(s)} \\

\hline

RMPGM & 250.1 & 6.260 & 196.5 & 7.752 & 357.7 & 20.620 \\

Inexact-RMPGM & 237.4 & 1.725 & 195.4 & 2.048 & 353.9 & 5.957 \\

TR-RMPGM & 76.4 & 1.506 & 56.5 & 1.742 & 56.6 & 2.778 \\

RMSD & 500.0 & 1.356 & 500.0 & 1.495 & 500.0 & 1.837 \\

\hline
\end{tabular}
\end{table}

\textbf{Results and Discussion.} 
Figure~\ref{fig:pareto_front} shows that the three RMPGM variants successfully handle both the sphere geometry and the nonsmooth $\ell_1$ regularization, producing well-distributed Pareto fronts across all problem sizes. In contrast, the RMSD baseline fails to approach the Pareto front and stagnates above the prescribed tolerance. As further shown in Table~\ref{tab:summary_sphere}, RMSD consistently exhausts the iteration budget, whereas all RMPGM variants converge reliably, highlighting the advantage of the proximal framework for nonsmooth multiobjective manifold optimization.
The table also indicates that Inexact-RMPGM requires a similar number of outer iterations to RMPGM, while achieving substantially lower computational time due to the reduced cost of solving the proximal subproblems inexactly.

Figures~\ref{fig:convergence} and~\ref{fig:time_comparison} illustrate the convergence behavior and runtime performance for the intermediate problem size $(n,m)=(256,100)$. RMPGM exhibits stable monotone descent, while TR-RMPGM converges substantially faster through adaptive updates of the penalty parameter $\sigma_k$. Inexact-RMPGM achieves convergence behavior comparable to RMPGM with reduced computational cost due to approximate subproblem solves. Overall, TR-RMPGM requires the fewest iterations and achieves the best runtime performance across all tested configurations.

\section{Conclusion}
\label{sec: Conclusion}
% We developed RMPGM, a first-order proximal gradient framework for composite multiobjective optimization on Riemannian manifolds. Global convergence to Pareto stationary points was established for general retractions, and an $\mathcal{O}(1/k)$ rate was proved. An inexact variant preserving global convergence and a trust-region variant that adapts the penalty parameter without prior knowledge of the Lipschitz constant were also analyzed. Future work includes local convergence analysis under the Riemannian KL property, sharper complexity results for the trust-region scheme, and extending the rate analysis to broader classes of retractions.

We developed RMPGM, a first-order proximal gradient framework for composite multiobjective optimization on Riemannian manifolds. Global convergence to Pareto stationary points was established under general retraction mappings, together with an $\mathcal{O}(1/k)$ convergence rate. An inexact variant preserving global convergence guarantees and a trust-region variant that adaptively adjusts the penalty parameter without requiring prior knowledge of the Lipschitz constant were also developed and analyzed.
Numerical results demonstrate that the proposed methods are robust and efficient. In particular, the trust-region variant consistently improves practical performance in terms of iteration and computational time, while maintaining stable convergence behavior.
Future work includes local convergence analysis under the Riemannian Kurdyka--Łojasiewicz property, sharper complexity bounds for the trust-region scheme, and extending the analysis to broader classes of retractions and more general nonsmooth structures.

\backmatter

\bmhead{Acknowledgments}
The author thanks Shiming Zhao and Tao Yan for helpful discussions on trust-region methods and for sharing reference code for single-objective optimization algorithms.

\section*{Declarations}
% \textbf{Funding} This work was supported by JST, the establishment of university fellowships towards the creation of science technology innovation (JPMJFS2123), and Japan Society for the Promotion of Science, Grant-in-Aid for Scientific Research (C) (19K11840) and Grant-in-Aid for Early-Career Scientists (20K14359).

\textbf{Conflict of interest} The author declares no competing interests.

% \textbf{Code availability} The code will be made available upon acceptance.

\bibliography{reference}

\end{document}